\documentclass[11pt]{amsart}%
\usepackage{amsmath}
\usepackage{amsfonts}
\usepackage{amssymb}
\usepackage{graphicx}%
\providecommand{\U}[1]{\protect\rule{.1in}{.1in}}
\newtheorem{Th}{Theorem}[section]
\newtheorem{Lem}[Th]{Lemma}
\newtheorem{Cor}[Th]{Corollary}
\newtheorem{Conj}[Th]{Conjecture}
\newtheorem{Prop}[Th]{Proposition}
\newtheorem{Def-Prop}[Th]{Definition-Proposition}

\theoremstyle{definition}

\newtheorem{Exa}[Th]{Example}
\newtheorem{Hyp}[Th]{Hypothesis}
\newtheorem{Rem}[Th]{Remark}

\allowdisplaybreaks

\begin{document}
\title{Parabolic restrictions and double deformations of weight multiplicities}

\begin{abstract}
We introduce some $(p,q)$-deformations of the weight multiplicities for the
representations of any simple Lie algebra $\mathfrak{g}$ over the complex
numbers. This is done by associating the indeterminate $q$ to the positive
roots of a parabolic subsystem of $\mathfrak{g}$ and the indeterminate $p$ to
the remaining positive roots. When $p=q$, we just recover the usual Lusztig
$q$-analogues of weight multiplicities. We then study the positivity of the
coefficients in these double deformations. In particular, the positivity holds
when $p=1$ in which case the polynomials have a natural algebraic
interpretation in terms of a parabolic Brylinski filtration. For the parabolic
restriction from type $C$ to type $A$, this positivity result was conjectured
by Lee. We also establish this positivity, in any finite type and for any $p$,
for a stabilized version of our double deformation. In addition, we study the
double deformation obtained by replacing the pair $(p,q)$ by $(p+1,q+1)$, show
it has nonnegative coefficients and admits a combinatorial description in
terms of crystals.

\end{abstract}
\author{C\'{e}dric Lecouvey}
\maketitle


\section{Introduction}

In this note, we introduce a double deformation of the weight multiplicities
for the representations of any simple Lie algebra $\mathfrak{g}$ over
$\mathbb{C}$ depending of two indeterminates $p$ and $q$. The deformation is
constructed from the $(p,q)$-Kostant partition function defined by the series%
\[
\prod_{\alpha\in R_{+}\setminus\overline{R}_{+}}\frac{1}{1-pe^{\alpha}}%
\prod_{\alpha\in\overline{R}_{+}}\frac{1}{1-qe^{\alpha}}=\sum_{\beta\in Q_{+}%
}\mathrm{P}_{p,q}(\beta)e^{\beta}.
\]
Here $R_{+}$ is the set of positive roots of $\mathfrak{g}$ and $\overline
{R}_{+}$ the set of positive roots corresponding to a parabolic subsystem,
that is the set of positive roots associated to a Levi subalgebra
$\overline{\mathfrak{g}}$ of $\mathfrak{g}$. For any dominant weight $\nu$ and
any weight $\mu$ ($\nu$ and $\mu$ are weights for $\mathfrak{g}$), we so get
polynomials $K_{\nu,\mu}(p,q)$ in $\mathbb{Z}[p,q]$ such that $K_{\nu,\mu
}(q,q)$ coincides with the ordinary Lusztig $q$-analogue \cite{Lut}.\ In
particular $K_{\nu,\mu}(1,1)$ gives the dimension of the weight space $\mu$ in
the irreducible $\mathfrak{g}$-module $V(\nu)$ of highest weight $\nu$.

In Section 3, we give a decomposition of the polynomial $K_{\nu,\mu}(p,q)$ in
terms of $p$-branching multiplicities corresponding to the restriction from
$\mathfrak{g}$ to $\overline{\mathfrak{g}}$ and ordinary Lusztig $q$-weight
multiplicities for $\overline{\mathfrak{g}}$ (Theorem \ref{Th_Dec_WLA}). This
permits to study the positivity of the coefficients of the polynomials
$K_{\nu,\mu}(p,q)$. We establish that $K_{\nu,\mu}(1,q)$ and $K_{\nu,\mu
}(0,q)$ belong to $\mathbb{N}[q]$ when $\nu$ and $\mu$ are dominant
weights.\ This proves in particular Conjecture 1.4 proposed by Lee in
\cite{lee} for the parabolic restriction from type $C_{n}$ to type $A_{n-1}%
$.\footnote{The first version of this note appeared on arXiv the 13 of
December 2024. This conjecture was also proved shortly after in the preprint
\cite{CKL} which appeared on the arXiv the 30 of December 2024.} Also, we
shows that the polynomials $K_{\nu,\mu}(1,q)$ coincide with the jump
polynomials corresponding to the parabolic Brylinski-Kostant filtration of the
weight space $\mu$ in $V(\nu)$ associated to a principal nilpotent element
$\overline{e}$ in $\overline{\mathfrak{g}}$. It is interesting here to also
mention that, in the classical cases and for their parabolic root subsystems
of type $A_{n-1}$, the polynomials $K_{\nu,\mu}(1,q)$ coincide with some row
shaped one-dimensional sums (see \cite{CKL} for the related definitions,
precise statements and proofs).

\noindent In Section 4, we explain how one can relax the hypothesis for $\mu$
to be dominant and prove, by using arguments due to Panyushev \cite{Pan}, that
the polynomial $K_{\nu,\mu}(p+1,q+1)$ belongs to $\mathbb{N}[p,q]$ for any
dominant weight $\lambda$ and any weight $\mu$ (Theorem \ref{Th_dooublePan}).
In Section 5, we get some positivity results for stabilized versions of the
polynomials $K_{\nu,\mu}(p,q)$. In the classical types and for the ordinary
Lusztig $q$-analogues, such stabilized versions were defined in \cite{Le061}
and \cite{Le06} and proved to coincide with one-dimensional sums in
\cite{LOS}.\ The stabilization was then obtained by translating the dominant
weights $\nu$ and $\mu$ by a multiple of the fundamental weight $\omega_{n}$.
Here we generalize the construction in both directions: first this is done for
our two parameters deformations and next we get stabilized polynomials for any
parabolic subsystem of a given root system (Theorem \ref{Th_stable}). Finally,
in Section $6$, we examine the question of the combinatorial description of
the weight multiplicities $K_{\nu,\mu}(1,q)$ in terms of crystal graphs. In
particular, this description is easy when $\overline{\mathfrak{g}}$ is maximal
of type $A$, that is its Dynkin diagram is obtained by removing in that of
$\mathfrak{g}$ the unique node for which $\overline{\mathfrak{g}}$ becomes of
type $A$. We also give a combinatorial description of the polynomials
$K_{\nu,\mu}(p+1,q+1)$ by introducing a polynomial charge statistics on the
vertices of the crystal $B(\nu)$ associated to the dominant weight $\nu
$.\ This construction is reminiscent to Lascoux-Sch\"{u}tzenberger's charge
statistics but in our setting, one associates a polynomial to each vertex of
$B(\nu)$ (for example a semistandard tableau in type $A$) and not only a
monomial. In particular the polynomials $K_{\nu,\mu}(q+1,q+1)=K_{\nu,\mu
}(q+1)$, which are just the Lusztig $q$-analogues evaluated in $q+1$ instead
of $q$, admit simple combinatorial descriptions for any weight $\mu$ even if
their algebraic interpretation is yet unclear. We conclude the section by
proposing some $(p,q)$-analogues of the Hall-Littlewood polynomials.

\section{Double deformation of weight multiplicities}

\label{Section_DD}In this Section, we recall some classical results on root
systems and the representation theory of the Lie algebras over $\mathbb{C}$.
We refer the reader to \cite{BBK}, \cite{FH}, or \cite{Hum} for a detailed
exposition. Consider such a finite-dimensional simple algebra $\mathfrak{g}$
and assume its root system $R$ is realized in the Euclidean space $E$. Its
Dynkin diagram is indexed by $I$ and we denote as usual by

\begin{itemize}
\item $R_{+}$ and $S=\{\alpha_{i},i\in S\}$ its subsets of positive and simple roots,

\item $W$ its Weyl group with generators $s_{i},i\in I$ associated to the
simple roots $\alpha_{i},i\in S$,

\item $Q$ its root lattice and $Q_{+}$ the cone generated by the positive roots,

\item $P$ its weight lattice and $P_{+}$ the cone of dominant weights
generated by the fundamental weights $\omega_{i},i\in I$,

\item $V(\nu)$ the simple $\mathfrak{g}$-module of highest weight $\nu\in
P_{+},$

\item $\langle\cdot,\cdot\rangle$ the usual Euclidean scalar product on $E$,

\item $\leq$ the dominance order on $P$ such that $\gamma\leq\mu$ if and only
if $\mu-\gamma\in Q_{+}$,

\item
\[
\mathfrak{g}=%
{\textstyle\bigoplus\limits_{\alpha\in R_{+}}}
\mathfrak{g}_{\alpha}\oplus\mathfrak{h}\oplus%
{\textstyle\bigoplus\limits_{\alpha\in R_{+}}}
\mathfrak{g}_{-\alpha}%
\]
the triangular decomposition of $\mathfrak{g}$ where $\mathfrak{h}$ is the
Cartan subalgebra.
\end{itemize}

\bigskip

Now, let $\overline{I}$ be a subset of $I$ and denote by $\overline{R}$ the
associated parabolic root system with Weyl group $\overline{W}$, set of
positive roots $\overline{R}_{+}$ and set of dominant weights $\overline
{P}_{+}$.\ Recall that each right coset $\overline{W}w$ admits a unique
minimal length element and write $U$ for the set of these minimal length
elements.\ Then each $w$ in $W$ decomposes in a unique way on the form
$W=\overline{w}u$ with $u\in U$ and $\overline{w}\in\overline{W}$.\ We
consider the Levi subalgebra%
\[
\overline{\mathfrak{g}}=%
{\textstyle\bigoplus\limits_{\alpha\in\overline{R}_{+}}}
\mathfrak{g}_{\alpha}\oplus\mathfrak{h}\oplus%
{\textstyle\bigoplus\limits_{\alpha\in\overline{R}_{+}}}
\mathfrak{g}_{-\alpha}.
\]
The algebra $\overline{\mathfrak{g}}$ is only semisimple and has the same
weight lattice as $\mathfrak{g}$.

Write $\overline{E}=\oplus_{i\in\overline{I}}\mathbb{R\omega}_{i}\mathrm{,}$
the linear real subspace of $E$ generated by the fundamental weights
$\omega_{i},i\in\overline{I}$ and set $\overline{E}^{\diamondsuit}%
=\oplus_{i\in I\setminus\overline{I}}\mathbb{R\omega}_{i}$. Each vector $v$ in
$E=\overline{E}\oplus\overline{E}^{\diamondsuit}$ decomposes uniquely on the
form
\begin{equation}
v=\overline{v}+v^{\diamondsuit} \label{OD}%
\end{equation}
where $\overline{v}\in\overline{E}$ and $v^{\diamondsuit}\in\overline
{E}^{\diamondsuit}$. Observe that the elements of $\overline{E}^{\diamondsuit
}$ are fixed under the action of $\overline{W}$ because $\langle\omega
_{i},\alpha_{j}\rangle=0$ for any $j\in\overline{I}$ and $i\in I\setminus
\overline{I}$.\ In fact we have
\[
\overline{E}^{\diamondsuit}\subset\{x\in E\mid\langle x,\alpha_{i}%
\rangle=0,i\in\overline{I}\}=\mathrm{span}(\alpha_{i},i\in\overline{I}%
)^{\perp}%
\]
which is the geometric orthogonal of the linear span of the simple roots
$\alpha_{i},i\in\overline{I}$.

\begin{Lem}
\label{Lem_stable_complement}For any $\overline{w}$ in $\overline{W}$, we have
$\overline{w}(R_{+}\setminus\overline{R}_{+})=R_{+}\setminus\overline{R}_{+}$,
that is the set $R_{+}\setminus\overline{R}_{+}$ is invariant under the action
of $\overline{W}$.
\end{Lem}

\begin{proof}
Consider $\alpha\in R_{+}\setminus\overline{R}_{+}$ and $i\in\overline{I}%
$.\ Then $s_{i}(\alpha)$ should belong to $R_{+}$.\ Indeed $\alpha_{i}$ is the
unique positive root which is changed into a negative root by $s_{i}$ and
$\alpha\neq\alpha_{i}$ by assumption.\ Therefore, we can set $s_{i}%
(\alpha)=\alpha^{\prime}$ with $\alpha^{\prime}$ in $R_{+}$. Now, we cannot
have $\alpha^{\prime}$ in $\overline{R}_{+}$ since otherwise $\alpha
=s_{i}(\alpha^{\prime})$ would belong to $\overline{R}$ which is stable under
the action of $\overline{W}$ and therefore we would have $\alpha\in
\overline{R}_{+}$ hence a contradiction.
\end{proof}

\bigskip

Each root in $R_{+}\setminus\overline{R}_{+}$ is nevertheless not fixed by the
action of $\overline{W}$: the action of an element of $\overline{W}$ will
permute the roots in $R_{+}\setminus\overline{R}_{+}$.

The simple $\overline{\mathfrak{g}}$-modules are labeled by the weights in the
cone $\overline{P}_{+}=\{\delta\in P\mid\langle\delta,\alpha_{i}\rangle
\geq0,i\in\overline{I}\}$. We have
\[
\overline{P}_{+}=%
{\textstyle\bigoplus\limits_{i\in\overline{I}}}
\mathbb{N}\omega_{i}\oplus%
{\textstyle\bigoplus\limits_{i\in I\setminus\overline{I}}}
\mathbb{Z}\omega_{i}%
\]
that is, each dominant weight in $\overline{P}_{+}$ can be written on the
form
\begin{equation}
\overline{\lambda}+\gamma\text{ with }\overline{\lambda}\in%
{\textstyle\bigoplus\limits_{i\in\overline{I}}}
\mathbb{N}\omega_{i}\text{ and }\gamma\in%
{\textstyle\bigoplus\limits_{i\in I\setminus\overline{I}}}
\mathbb{Z}\omega_{i}. \label{DomWeightPar}%
\end{equation}
where $\gamma$ is fixed under the action of $\overline{W}$. We write
$\overline{V}(\overline{\lambda}+\gamma)$ for the simple $\overline
{\mathfrak{g}}$-module with highest weight $\overline{\lambda}+\gamma$.

\begin{Exa}
Assume $\mathfrak{g}=\mathfrak{sp}_{2n}$ is the Lie algebra of type $C_{n}$
and $\overline{I}=\{1,\ldots,n-1\}$ (the longest simple root is $\alpha_{n}$).
With the usual realization of the root system of type $C_{n}$ (see
\cite{BBK}), we have $E=\oplus_{i=1}^{n}\mathbb{R\varepsilon}_{i},$
$\alpha_{i}=\varepsilon_{i}-\varepsilon_{i+1},i=1,\ldots,n-1$ and $\alpha
_{n}=2\varepsilon_{n}$. Moreover $\omega_{i}=\varepsilon_{1}+\cdots
+\varepsilon_{i}$. Then $\overline{\mathfrak{g}}$ is isomorphic to
$\mathfrak{gl}_{n}$ and we have
\[
\overline{E}=%
{\textstyle\bigoplus_{i=1}^{n-1}}
\mathbb{R\omega}_{i}=%
{\textstyle\bigoplus_{i=1}^{n-1}}
\mathbb{R\varepsilon}_{i}.
\]
The dominant weights of $\overline{\mathfrak{g}}$ are the generalized
partitions $\kappa=(\kappa_{1}\geq\cdots\geq\kappa_{n})\in\mathbb{Z}^{n}$ and
we can write write%
\[
\kappa=\overline{\lambda}+\gamma
\]
with $\overline{\lambda}=(\kappa_{1}-\kappa_{n},\ldots,\kappa_{n-1}-\kappa
_{n},0)$ and $\gamma=\kappa_{n}\omega_{n}$.
\end{Exa}

\bigskip

Consider $q$ and $p$ two formal indeterminates and introduce the $p,q$-Kostant
partitions $\mathrm{P}_{p,q}$ defined from the expansion of the formal series%
\begin{equation}
\prod_{\alpha\in R_{+}\setminus\overline{R}_{+}}\frac{1}{1-pe^{\alpha}}%
\prod_{\alpha\in\overline{R}_{+}}\frac{1}{1-qe^{\alpha}}=\sum_{\beta\in Q_{+}%
}\mathrm{P}_{p,q}(\beta)e^{\beta}. \label{(p,q)-Kostant}%
\end{equation}
When $p=q$, this is the ordinary $q$-partition function (see \cite{kost}).
Given two dominant weights $\nu$ and $\mu$ in $P_{+}$, set%
\begin{equation}
K_{\nu,\mu}(p,q)=\sum_{w\in W}\varepsilon(w)\mathrm{P}_{p,q}(w(\lambda
+\rho)-\mu-\rho) \label{Def_Double}%
\end{equation}
where $\rho=\frac{1}{2}\sum_{\alpha\in R_{+}}\alpha$. When $p=q$, the
polynomial $K_{\nu,\mu}(q,q)$ is an ordinary Lusztig $q$-analogue. In
particular, when $p=q=1$, the integer $K_{\nu,\mu}=K_{\nu,\mu}(1,1)$ gives the
dimension of the weight space
\[
V(\nu)_{\mu}=\{v\in V(\nu)\mid h(v)=\mu(v)v,h\in\mathfrak{h}\}.
\]

\begin{Rem}
We have $w(\lambda+\rho)-\mu-\rho=w(\lambda+\rho)-(\lambda+\rho)+(\lambda
-\mu)$ and $(\lambda+\rho)-w(\lambda+\rho)$ belongs to $Q_{+}$. Also for any
weight $\gamma\in P$, we have $\mathrm{P}_{p,q}(\gamma)\neq0$ only if
$\gamma\in Q_{+}$. Therefore, the polynomial $K_{\nu,\mu}(p,q)$ is nonzero
only if $\mu\leq\lambda$.
\end{Rem}

Let us now define the partition function $\mathrm{\hat{P}}_{p}$ from the
expansion of the formal series%
\[
\prod_{\alpha\in R_{+}\setminus\overline{R}_{+}}\frac{1}{1-pe^{\alpha}}%
=\sum_{\eta\in\hat{Q}_{+}}\mathrm{\hat{P}}_{p}(\eta)e^{\eta}%
\]
where $\hat{Q}_{+}=\sum_{\alpha\in R_{+}\setminus\overline{R}_{+}%
}\mathbb{N\alpha}$. Since we have $\overline{w}(R_{+}\setminus\overline{R}%
_{+})=R_{+}\setminus\overline{R}_{+}$ for any $\overline{w}\in\overline{W}$,
the cone $\hat{Q}_{+}$ is also stable under the action of $\overline{W}$.
Moreover, for any $\overline{w}\in\overline{W}$ and any $\eta\in\hat{Q}_{+}$,
we have $\mathrm{\hat{P}}_{p}(\overline{w}(\eta))=\mathrm{\hat{P}}_{p}(\eta)$.
This is because $\overline{w}$ permutes the roots in $R_{+}\setminus
\overline{R}_{+}.\ $Given a dominant weight $\nu\in P_{+}$ and a parabolic
dominant weight $\overline{\lambda}+\gamma\in\overline{P}_{+}$ as in
(\ref{DomWeightPar}) define the polynomial%
\[
b_{\nu,\overline{\lambda}+\gamma}(p)=\sum_{w\in W}\varepsilon(w)\mathrm{\hat
{P}}_{p}(w(\nu+\rho)-\overline{\lambda}-\gamma-\rho).
\]
Then, it is classical to deduce from the Weyl character formula (see for
example \cite{GW}) that $b_{\nu,\overline{\lambda}+\gamma}(1)$ gives the
multiplicity of $\overline{V}(\overline{\lambda}+\gamma)$ in the restriction
of $V(\nu)$ from $\mathfrak{g}$ to $\overline{\mathfrak{g}}$. Observe that in
general, the polynomial $b_{\nu,\overline{\lambda}+\gamma}(p)$ does not belong
to $\mathbb{N}[p]$. We nevertheless mention the following conjecture which, as
far as we are aware, is only partially proved (in particular by Broer
\cite{Br} in type $A$).\ We also refer to \cite{Pan2}, \cite{Stem} and the
references therein for more details and results on this problem.

\begin{Conj}
\label{Conj_Branch}The polynomial $b_{\nu,\overline{\lambda}+\gamma}(p)$ has
nonnegative integer coefficients at least when $\overline{\lambda}+\gamma\in
P_{+}$, that is when $\overline{\lambda}+\gamma$ is a dominant weight for
$\mathfrak{g}$ (and not only for $\overline{\mathfrak{g}}$).
\end{Conj}

Write $\overline{\mathrm{P}}_{q}$ for the usual partition function associated
to the root system $\overline{R}$ defined by
\[
\prod_{\alpha\in\overline{R}_{+}}\frac{1}{1-qe^{\alpha}}=\sum_{\delta
\in\overline{Q}_{+}}\overline{\mathrm{P}}_{q}(\delta)e^{\delta}.
\]
Given two dominant weights $\overline{\lambda}+\gamma$ and $\overline{\mu
}+\gamma^{\prime}$ in $\overline{P}_{+}$ written as in (\ref{DomWeightPar}),
we can consider the Lusztig $q$-analogue for the subalgebra $\overline
{\mathfrak{g}}$%
\[
\overline{K}_{\overline{\lambda}+\gamma,\overline{\mu}+\gamma^{\prime}%
}(q)=\sum_{\overline{w}\in\overline{W}}\varepsilon(\overline{w})\overline
{\mathrm{P}}_{q}(\overline{w}(\overline{\lambda}+\gamma+\overline{\rho
})-\overline{\mu}-\gamma^{\prime}-\overline{\rho}).
\]
Since $\overline{w}(\gamma)=\gamma$ for any $\overline{w}\in\overline{W}$, we
have $\overline{w}(\overline{\lambda}+\gamma+\overline{\rho})-\overline{\mu
}-\gamma^{\prime}-\overline{\rho}=\overline{w}(\overline{\lambda}%
+\overline{\rho})-\overline{\mu}-\overline{\rho}+\gamma-\gamma^{\prime}$. Also
it follows from the previous definition that $\overline{w}(\overline{\lambda
}+\overline{\rho})-\overline{\mu}-\overline{\rho}$ belongs to $\overline{E}$
whereas $\gamma-\gamma^{\prime}$ belongs to $\overline{E}^{\diamondsuit}$.
Moreover, we have $\overline{\mathrm{P}}_{q}(\beta)\neq0$ only if $\beta$
belongs to $\overline{E}$. Thus, we can conclude that $\overline{K}%
_{\overline{\lambda}+\gamma,\overline{\mu}+\gamma^{\prime}}(q)$ is nonzero
only when $\gamma=\gamma^{\prime}$ in which case we moreover have%
\begin{equation}
\overline{K}_{\overline{\lambda}+\gamma,\overline{\mu}+\gamma}(q)=\overline
{K}_{\overline{\lambda},\overline{\mu}}(q). \label{simply_Kbar}%
\end{equation}
We can also enlarge the definition of the polynomials $\overline{K}%
_{\overline{\lambda},\overline{\mu}}(q)$ by replacing the dominant weight
$\overline{\lambda}\in%
{\textstyle\bigoplus\limits_{i\in\overline{I}}}
\mathbb{N}\overline{\omega}_{i}$ by any weight $\overline{\kappa}\in%
{\textstyle\bigoplus\limits_{i\in\overline{I}}}
\mathbb{Z}\overline{\omega}_{i}$. Then we get $\overline{K}_{\overline{\kappa
},\overline{\mu}}(q)=0$ or there exists $\overline{w}$ in $\overline{W}$ and
$\overline{\lambda}\in\overline{P}_{+}$ such that
\[
\overline{w}^{-1}(\overline{\lambda}+\overline{\rho})-\overline{\rho
}=\overline{\kappa}\text{ and }\overline{K}_{\overline{\kappa},\overline{\mu}%
}(q)=\varepsilon(\overline{w})\overline{K}_{\overline{\lambda},\overline{\mu}%
}(q)\text{.}%
\]

\bigskip

Finally, by elementary properties of formal series, we must have for any
$\beta\in Q_{+}$%
\begin{equation}
\mathrm{P}_{p,q}(\beta)=\sum_{\gamma\in\hat{Q}_{+},\delta\in\overline{Q}%
_{+}\mid\gamma+\delta=\beta}\mathrm{\hat{P}}_{p}(\gamma)\overline{\mathrm{P}%
}_{q}(\delta)=\sum_{\gamma\in\hat{Q}_{+}}\mathrm{\hat{P}}_{p}(\gamma
)\overline{\mathrm{P}}_{q}(\beta-\gamma). \label{Cauchy}%
\end{equation}
We also recall the Weyl character formula. For each dominant weight $\nu$ in
$P_{+}$, the character of $V(\nu)$ is the polynomial in $\mathbb{Z}^{W}[P]$
such that%
\[
s_{\nu}=\frac{\sum_{w\in W}\varepsilon(w)e^{w(\nu+\rho)-\rho}}{\prod
_{\alpha\in R_{+}}(1-e^{-\alpha})}.
\]
In particular for $\nu=0$, we get the Weyl denominator formula%
\[
\prod_{\alpha\in R_{+}}(1-e^{-\alpha})=\sum_{w\in W}\varepsilon(w)e^{w(\rho
)-\rho}.
\]
We shall also use the notation $a_{\gamma}=\sum_{w\in W}\varepsilon
(w)e^{w(\gamma)}$ for any $\gamma\in P$. Then $s_{\lambda}=\frac
{a_{\lambda+\rho}}{a_{\rho}}$.

\section{Decomposition of the double deformation}

In this section, we establish a decomposition theorem for our double
deformations of weight multiplicities. When $p=q=1$, it just means that one
can obtain the restriction of any representation of $\mathfrak{g}$ to its
Cartan subalgebra $\mathfrak{h}$ by performing first its restriction to
$\overline{\mathfrak{g}}$ and next the restriction from $\overline
{\mathfrak{g}}$ to $\mathfrak{h}$. When we have only $p=1$, we interpret the
polynomials $K_{\nu,\mu}(1,q)$ in terms of the parabolic Brylinski filtration
of $V(\nu)$ obtained from $\overline{\mathfrak{g}}$.

\subsection{The general case}

\begin{Th}
\label{Th_Dec_WLA}For any dominant weight $\lambda$ in $P_{+}$ and any weight
$\mu$ in $P$ we have the decomposition%
\[
K_{\nu,\mu}(p,q)=\sum_{\overline{\lambda}\in\overline{P}_{+}}b_{\nu
,\overline{\lambda}+\mu^{\diamondsuit}}(p)\overline{K}_{\overline{\lambda
},\overline{\mu}}(q)=\sum_{\overline{\lambda}\in\overline{P}_{+}}%
b_{\nu,\overline{\lambda}+\mu^{\diamondsuit}}(p)\overline{K}_{\overline
{\lambda}+\mu^{\diamondsuit},\mu}(q)
\]
where $\overline{\lambda}+\mu^{\diamondsuit}$ belongs to $\overline{P}_{+}$.
\end{Th}

\begin{proof}
By definition of the polynomial $K_{\nu,\mu}(p,q)$ and by using (\ref{Cauchy}%
), we obtain
\[
K_{\nu,\mu}(p,q)=\sum_{u\in U}\varepsilon(u)\sum_{\gamma\in\hat{Q}_{+}%
}\mathrm{\hat{P}}_{p}(\gamma)\sum_{\overline{w}\in\overline{W}}\varepsilon
(\overline{w})\overline{\mathrm{P}}_{q}(\overline{w}u(\nu+\rho)-\mu
-\rho-\gamma)
\]
where we have decomposed each $w$ in $W$ on the form $w=\overline{w}u$ with
$u\in U$ and $\overline{w}\in\overline{W}$. Since $\hat{Q}_{+}$ is stable
under the action of $\overline{W}$ and $\mathrm{\hat{P}}_{p}(\overline{w}%
^{-1}(\gamma))=\mathrm{\hat{P}}_{p}(\gamma)$ for any $\overline{w}\in
\overline{W}$ and $\gamma\in\hat{Q}_{+}$, we also have
\[
K_{\nu,\mu}(p,q)=\sum_{u\in U}\varepsilon(u)\sum_{\gamma^{\prime}\in\hat
{Q}_{+}}\mathrm{\hat{P}}_{p}(\gamma^{\prime})\sum_{\overline{w}\in\overline
{W}}\varepsilon(\overline{w})\overline{\mathrm{P}}_{q}(\overline{w}u(\nu
+\rho)-\mu-\rho-\overline{w}(\gamma^{\prime}))
\]
by setting $\gamma^{\prime}=\overline{w}^{-1}(\gamma)$.

We can now apply decompositions of type (\ref{OD}) to get $\rho=\overline
{\rho}+\rho^{\diamondsuit}$ and $\mu=\overline{\mu}+\mu^{\diamondsuit}$. By
using that $\overline{w}(\rho^{\diamondsuit})=\rho^{\diamondsuit}$ and
$\overline{w}(\mu^{\diamondsuit})=\mu^{\diamondsuit},$ we obtain%
\begin{multline*}
\overline{w}\left(  u(\nu+\rho)\right)  -\mu-\rho-\overline{w}(\gamma^{\prime
})=\overline{w}[u(\nu+\rho)-\mu^{\diamondsuit}-\rho^{\diamondsuit}%
-\gamma^{\prime}]-\overline{\rho}-\overline{\mu}=\\
\overline{w}[(u(\nu+\rho)-\mu^{\diamondsuit}-\rho^{\diamondsuit}%
-\gamma^{\prime}-\overline{\rho})+\overline{\rho}]-\overline{\rho}%
-\overline{\mu}=\\
\overline{w}[(u(\nu+\rho)-\mu^{\diamondsuit}-\rho-\gamma^{\prime}%
)+\overline{\rho}]-\overline{\rho}-\overline{\mu}.
\end{multline*}
This gives
\[
\overline{\mathrm{P}}_{q}(\overline{w}u(\nu+\rho)-\mu-\rho-\overline{w}%
(\gamma^{\prime}))=\overline{\mathrm{P}}_{q}\left(  \overline{w}[(u(\nu
+\rho)-\rho-\mu^{\diamondsuit}-\gamma^{\prime})+\overline{\rho}]-\overline
{\rho}-\overline{\mu}\right)  .
\]

Therefore, we have
\[
\sum_{\gamma^{\prime}\in\hat{Q}_{+}}\mathrm{\hat{P}}_{p}(\gamma^{\prime}%
)\sum_{\overline{w}\in\overline{W}}\varepsilon(\overline{w})\overline
{\mathrm{P}}_{q}(\overline{w}u(\nu+\rho)-\mu-\rho-\overline{w}(\gamma^{\prime
}))=\sum_{\gamma^{\prime}\in\hat{Q}_{+}}\mathrm{\hat{P}}_{p}(\gamma^{\prime
})\overline{K}_{u(\nu+\rho)-\rho-\mu^{\diamondsuit}-\gamma^{\prime}%
,\overline{\mu}}(q).
\]
This yields
\[
K_{\nu,\mu}(p,q)=\sum_{u\in U}\varepsilon(u)\sum_{\gamma^{\prime}\in\hat
{Q}_{+}}\mathrm{\hat{P}}_{p}(\gamma^{\prime})\overline{K}_{u(\nu+\rho
)-\rho-\mu^{\diamondsuit}-\gamma^{\prime},\overline{\mu}}(q).
\]
For each nonzero polynomial $\overline{K}_{u(\nu+\rho)-\rho-\mu^{\diamondsuit
}-\gamma^{\prime},\overline{\mu}}(q)$ there exists $\overline{w}$ in
$\overline{W}$ and $\overline{\lambda}\in\overline{P}_{+}$ such that
\[
\overline{w}^{-1}(\overline{\lambda}+\overline{\rho})-\overline{\rho}%
=u(\nu+\rho)-\rho-\mu^{\diamondsuit}-\gamma^{\prime}\text{ and }\overline
{K}_{u(\nu+\rho)-\rho-\mu^{\diamondsuit}-\gamma^{\prime},\overline{\mu}%
}(q)=\varepsilon(\overline{w})\overline{K}_{\overline{\lambda},\overline{\mu}%
}(q).
\]
We then have
\[
\gamma^{\prime}=u(\nu+\rho)-\rho-\mu^{\diamondsuit}-\overline{w}%
^{-1}(\overline{\lambda}+\overline{\rho})+\overline{\rho}=u(\nu+\rho
)-\rho^{\diamondsuit}-\mu^{\diamondsuit}-\overline{w}^{-1}(\overline{\lambda
}+\overline{\rho}).
\]
By gathering all the contributions corresponding to the same dominant weight
$\overline{\lambda}$, we thus obtain
\[
K_{\nu,\mu}(p,q)=\sum_{\overline{\lambda}\in\overline{P}_{+}}\left(
\sum_{u\in U}\varepsilon(u)\sum_{\overline{w}\in\overline{W}}\varepsilon
(\overline{w})\mathrm{\hat{P}}_{p}\left(  u(\nu+\rho)-\rho^{\diamondsuit}%
-\mu^{\diamondsuit}-\overline{w}^{-1}(\overline{\lambda}+\overline{\rho
})\right)  \right)  \overline{K}_{\overline{\lambda},\overline{\mu}}(q).
\]
Recall that for any $\eta\in\hat{Q}_{+}$ and any $\overline{w}\in\overline{W}%
$, we have $\mathrm{\hat{P}}_{p}(\overline{w}(\eta))=\mathrm{\hat{P}}_{p}%
(\eta)$. We also have $\overline{w}(\rho^{\diamondsuit})=\rho^{\diamondsuit}$
and $\overline{w}(\mu^{\diamondsuit})=\mu^{\diamondsuit}$ and thus%
\begin{multline*}
\mathrm{\hat{P}}_{p}\left(  u(\nu+\rho)-\rho^{\diamondsuit}-\mu^{\diamondsuit
}-\overline{w}^{-1}(\overline{\lambda}+\overline{\rho})\right)  =\mathrm{\hat
{P}}_{p}\left(  \overline{w}u(\nu+\rho)-\rho^{\diamondsuit}-\mu^{\diamondsuit
}-\overline{\lambda}-\overline{\rho})\right) \\
=\mathrm{\hat{P}}_{p}\left(  \overline{w}u(\nu+\rho)-\mu^{\diamondsuit
}-\overline{\lambda}-\rho)\right)  .
\end{multline*}
This permits to write%
\begin{multline}
K_{\nu,\mu}(p,q)=\sum_{\overline{\lambda}\in\overline{P}_{+}}\left(
\sum_{u\in U}\sum_{\overline{w}\in\overline{W}}\varepsilon(\overline
{w}u)\mathrm{\hat{P}}_{p}\left(  \overline{w}u(\nu+\rho)-\mu^{\diamondsuit
}-\overline{\lambda}-\rho)\right)  \right)  \overline{K}_{\overline{\lambda
},\overline{\mu}}(q)=\label{fund}\\
\sum_{\overline{\lambda}\in\overline{P}_{+}}\left(  \sum_{w\in W}%
\varepsilon(w)\mathrm{\hat{P}}_{p}\left(  w(\nu+\rho)-\mu^{\diamondsuit
}-\overline{\lambda}-\rho)\right)  \right)  \overline{K}_{\overline{\lambda
},\overline{\mu}}(q)=\sum_{\overline{\lambda}\in\overline{P}_{+}}%
b_{\nu,\overline{\lambda}+\mu^{\diamondsuit}}(p)\overline{K}_{\overline
{\lambda},\overline{\mu}}(q)
\end{multline}
as claimed. It is straightforward to check that $\overline{\lambda}%
+\mu^{\diamondsuit}$ belongs to $\overline{P}_{+}$ since $\mu^{\diamondsuit}$
belongs to $\overline{E}^{\diamondsuit}$ and thus $\langle\overline{\lambda
}+\mu^{\diamondsuit},\alpha_{i}\rangle=\langle\overline{\lambda},\alpha
_{i}\rangle\geq0$ for any $i\in\overline{I}$. The last equality of the theorem
follows immediately from Equality (\ref{simply_Kbar}) since we have
$\mu=\overline{\mu}+\mu^{\diamondsuit}$.
\end{proof}

\subsection{The case $p=1$ and the parabolic Brylinski-Kostant filtration}

\label{subsec_p=1}When $p=1,$ we obtain immediately from Theorem
\ref{Th_Dec_WLA} the decomposition
\begin{equation}
K_{\nu,\mu}(1,q)=\sum_{\overline{\lambda}\in\overline{P}_{+}\cap\overline{E}%
}b_{\nu,\overline{\lambda}+\mu^{\diamondsuit}}\overline{K}_{\overline{\lambda
}+\mu^{\diamondsuit},\mu}(q) \label{K(1,q)}%
\end{equation}
where the nonnegative integer $b_{\nu,\overline{\lambda}+\mu^{\diamondsuit}}$
is the branching coefficient giving the multiplicity of the $\overline
{\mathfrak{g}}$-module $\overline{V}(\overline{\lambda}+\mu^{\diamondsuit})$
in the restriction of $V(\nu)$ from $\mathfrak{g}$ to $\overline{\mathfrak{g}%
}$. In particular, when $\overline{\mu}$ belongs to $\overline{P}_{+}$, we get
that $K_{\nu,\mu}(1,q)$ belongs to $\mathbb{N}[q]$ since each polynomial
$\overline{K}_{\overline{\lambda}+\mu^{\diamondsuit},\mu}(q)=\overline
{K}_{\overline{\lambda},\overline{\mu}}(q)$ does. It is then possible to give
an algebraic interpretation of the polynomial $K_{\nu,\mu}(1,q)$ in terms of
the parabolic Brylinski-Kostant filtration corresponding to $\overline
{\mathfrak{g}}$.\ To do so, denote by $e_{i},i\in I$ the raising Chevalley
generators\footnote{We assume here that we choose a set of Chevalley
generators compatible with the triangular decomposition of $\mathfrak{g}$
considered.} of the Lie algebra $\mathfrak{g}$ and set%
\[
\overline{e}=\sum_{i\in\overline{I}}e_{i}.
\]
Then the element $\overline{e}$ is a principal nilpotent for the Lie algebra
$\overline{\mathfrak{g}}$.

Now assume that $\nu\in P_{+}$ is fixed. For any nonnegative integer $k$,
consider the increasing sequence (for the inclusion) of linear subspaces%
\[
\overline{J}^{(k)}(\nu)_{\mu}:=\{v\in V(\nu)_{\mu}\mid\overline{e}%
^{(k)}(v)=0\}.
\]
The space $\overline{J}^{(k)}(\nu)$ is thus the intersection of the weight
space $V(\nu)_{\mu}$ with the kernel of the action of the element
$\overline{e}^{k}$ on the $\mathfrak{g}$-module $V(\nu)$.\ We can define the
jump polynomial $\mathcal{K}_{\nu,\mu}(q)$ by setting%
\[
\mathcal{K}_{\nu,\mu}(q)=\sum_{k\geq0}\left(  \dim\left(  \overline{J}%
^{(k+1)}(\nu)_{\mu}\right)  -\dim\left(  \overline{J}^{(k)}(\nu)_{\mu}\right)
\right)  q^{k}.
\]

\begin{Th}
For any $\nu\in P_{+}$ and any $\mu=\overline{\mu}+\mu^{\diamondsuit}%
\in\overline{P}_{+}$, we have
\[
\mathcal{K}_{\nu,\mu}(q)=K_{\nu,\mu}(1,q).
\]

\end{Th}

\begin{proof}
The $\mathfrak{g}$-module $V(\nu)$ decomposes as a direct sum of irreducible
$\overline{\mathfrak{g}}$-modules. Let $\overline{M}$ be one of these
$\overline{\mathfrak{g}}$-modules.\ Assume that $\overline{M}$ has highest
weight $\overline{\lambda}+\gamma\in\overline{P}_{+}$ (decomposed as in
(\ref{DomWeightPar})) and $\overline{M}\cap V(\nu)_{\mu}\neq\{0\}$. Then, the
difference $\overline{\lambda}+\gamma-\mu$ is a sum of simple roots
$\alpha_{i},i\in\overline{I}$ and therefore belongs to $\overline{E}$. It
follows that $\overline{\lambda}+\gamma$ and $\mu$ have the same component on
$\overline{E}^{\diamondsuit}$, that is $\gamma=\mu^{\diamondsuit}$. Recall
that the decomposition of $V(\nu)$ into its weight spaces refines the
decomposition of $V(\nu)$ into its irreducible $\overline{\mathfrak{g}}%
$-components. By considering only the irreducible $\overline{\mathfrak{g}}%
$-modules with highest weight of the form $\overline{\lambda}+\mu
^{\diamondsuit},\overline{\lambda}\in\overline{P}_{+}\cap\overline{E}$ and
their multiplicities $b_{\nu,\overline{\lambda}+\mu^{\diamondsuit}}$ in the
decomposition of $V(\nu)$, we therefore get the isomorphism%
\[
V(\nu)_{\mu}\simeq%
{\textstyle\bigoplus\limits_{\overline{\lambda}\in\overline{P}_{+}%
\cap\overline{E}}}
\left(  \overline{V}(\overline{\lambda}+\mu^{\diamondsuit})\cap V(\nu)_{\mu
}\right)  ^{\oplus b_{\nu,\overline{\lambda}+\mu^{\diamondsuit}}}=%
{\textstyle\bigoplus\limits_{\overline{\lambda}\in\overline{P}_{+}%
\cap\overline{E}}}
\overline{V}(\overline{\lambda}+\mu^{\diamondsuit})_{\mu}^{\oplus
b_{\nu,\overline{\lambda}+\mu^{\diamondsuit}}}%
\]
where $\overline{V}(\overline{\lambda}+\mu^{\diamondsuit})_{\mu}$ is the
weight space associated to $\mu$ in $\overline{V}(\overline{\lambda}%
+\mu^{\diamondsuit})$. Since each $\overline{\mathfrak{g}}$-module is stable
under the action of $\overline{e}$, this also gives for any $k\geq0$%
\[
\overline{J}^{(k)}(\nu)_{\mu}\simeq%
{\textstyle\bigoplus\limits_{\overline{\lambda}\in\overline{P}_{+}%
\cap\overline{E}}}
\left(  \overline{V}(\overline{\lambda}+\mu^{\diamondsuit})_{\mu}\cap
\overline{J}^{(k)}(\nu)_{\mu}\right)  ^{\oplus b_{\nu,\overline{\lambda}%
+\mu^{\diamondsuit}}}%
\]
and
\[
\dim\left(  \overline{J}^{(k)}(\nu)_{\mu}\right)  =\sum_{\overline{\lambda}%
\in\overline{P}_{+}\cap\overline{E}}b_{\nu,\overline{\lambda}+\mu
^{\diamondsuit}}\dim\left(  \overline{V}(\overline{\lambda}+\mu^{\diamondsuit
})_{\mu}\cap\overline{J}^{(k)}(\nu)_{\mu}\right)  .
\]
We so obtain%
\begin{multline*}
\mathcal{K}_{\nu,\mu}(q)=\\
\sum_{k\geq0}\sum_{\overline{\lambda}\in\overline{P}_{+}\cap\overline{E}%
}b_{\nu,\overline{\lambda}+\mu^{\diamondsuit}}\left(  \dim\left(  \overline
{V}(\overline{\lambda}+\mu^{\diamondsuit})_{\mu}\cap\overline{J}^{(k+1)}%
(\nu)_{\mu}\right)  -\dim\left(  \overline{V}(\overline{\lambda}%
+\mu^{\diamondsuit})_{\mu}\cap\overline{J}^{(k)}(\nu)_{\mu}\right)  \right)
q^{k}\\
=\sum_{\overline{\lambda}\in\overline{P}_{+}\cap\overline{E}}b_{\nu
,\overline{\lambda}+\mu^{\diamondsuit}}\sum_{k\geq0}\left(  \dim\left(
\overline{V}(\overline{\lambda}+\mu^{\diamondsuit})_{\mu}\cap\overline
{J}^{(k+1)}(\nu)_{\mu}\right)  -\dim\left(  \overline{V}(\overline{\lambda
}+\mu^{\diamondsuit})_{\mu}\cap\overline{J}^{(k)}(\nu)_{\mu}\right)  \right)
q^{k}.
\end{multline*}
For any $\overline{\lambda}\in\overline{P}_{+}\cap\overline{E}$, Brylinski's
theorem (see \cite{bry}) equates the jump polynomial in $\overline
{V}(\overline{\lambda}+\mu^{\diamondsuit})$ for the weight $\mu$ with the
parabolic Lusztig's $q$-analogue $\overline{K}_{\overline{\lambda}%
+\mu^{\diamondsuit},\mu}(q)$ so that
\[
\overline{K}_{\overline{\lambda}+\mu^{\diamondsuit},\mu}(q)=\sum_{k\geq
0}\left(  \dim\left(  \overline{V}(\overline{\lambda}+\mu^{\diamondsuit}%
)_{\mu}\cap\overline{J}^{(k+1)}(\nu)_{\mu}\right)  -\dim\left(  \overline
{V}(\overline{\lambda}+\mu^{\diamondsuit})_{\mu}\cap\overline{J}^{(k)}%
(\nu)_{\mu}\right)  \right)  q^{k}.
\]
Finally, we have
\[
\mathcal{K}_{\nu,\mu}(q)=\sum_{\overline{\lambda}\in\overline{P}_{+}%
\cap\overline{E}}b_{\nu,\overline{\lambda}+\mu^{\diamondsuit}}\overline
{K}_{\overline{\lambda}+\mu^{\diamondsuit},\mu}(q)=K_{\nu,\mu}(1,q)
\]
by (\ref{K(1,q)}).
\end{proof}

\subsection{Other positivity results}

We summarize in this paragraph what can be immediately claimed about the
nonnegativity of the coefficients of the polynomials $K_{\nu,\mu}(p,q)$ (whose
coefficients belong a priori to $\mathbb{Z}$). We get the following direct
corollary of our Theorem \ref{Th_Dec_WLA}.

\begin{Cor}
Consider $\lambda$ a dominant weight for $\mathfrak{g}$. We have for any
weight $\mu\in P$%
\begin{align*}
K_{\nu,\mu}(1,q)  &  =\sum_{\overline{\lambda}\in\overline{P}_{+}}%
b_{\nu,\overline{\lambda}+\mu^{\diamondsuit}}(1)\overline{K}_{\overline
{\lambda},\overline{\mu}}(q),\\
K_{\nu,\mu}(0,q)  &  =\overline{K}_{\overline{\nu},\overline{\mu}}(q),\\
K_{\nu,\mu}(p,0)  &  =b_{\nu,\mu}(p),\\
K_{\nu,\mu}(p,p)  &  =K_{\nu,\mu}(p).
\end{align*}
Moreover,

\begin{itemize}
\item the polynomial $K_{\nu,\mu}(1,q)$ has nonnegative integer coefficients
when $\mu\in\overline{P}_{+}$,

\item the polynomial $K_{\nu,\mu}(q,q)$ has nonnegative integer coefficients
when $\mu\in P_{+}$,

\item the polynomial $K_{\nu,\mu}(p,0)$ has nonnegative in the cases where
Conjecture \ref{Conj_Branch} holds.
\end{itemize}
\end{Cor}

\begin{proof}
The case $p=1$ has already been studied in \S \ \ref{subsec_p=1}.

Now when $p=0$, we have $\mathrm{\hat{P}}_{p}(\gamma)=\delta_{1,\gamma}$ for
any $\gamma\in\hat{Q}_{+}$.\ Therefore for any $\overline{\lambda}\in
\overline{P}_{+}$, we get $b_{\nu,\overline{\lambda}+\mu^{\diamondsuit}%
}(0)\neq0$ if and only if $\nu=\overline{\lambda}+\mu^{\diamondsuit}$. This
means that $\overline{\lambda}=\overline{\nu}$ and $\mu^{\diamondsuit}%
=\nu^{\diamondsuit}$. In this case we get the equality $b_{\nu,\overline{\nu
}+\nu^{\diamondsuit}}(0)=b_{\nu,\nu}(0)=1$. This therefore gives
\[
K_{\nu,\mu}(0,q)=\overline{K}_{\overline{\nu},\overline{\mu}}(q).
\]

When $q=0$, we have $\overline{K}_{\overline{\lambda},\overline{\mu}}(0)\neq0$
only when $\overline{\lambda}=\overline{\mu}$ and then $\overline
{K}_{\overline{\lambda},\overline{\mu}}(0)=1$.\ Thus by Theorem
\ref{Th_Dec_WLA}, we get%
\[
K_{\nu,\mu}(p,0)=b_{\nu,\overline{\mu}+\mu^{\diamondsuit}}(p)=b_{\nu,\mu}(p)
\]
which has nonnegative coefficients in the cases where Conjecture
\ref{Conj_Branch} holds.

The last identity $K_{\nu,\mu}(q,q)=K_{\nu,\mu}(q)$ follows directly from the
definition of our polynomials $K_{\nu,\mu}(p,q)$.
\end{proof}

$\bigskip$

\begin{Rem}
When $p=1$, $\mathfrak{g}$ is one of the classical Lie algebras of type
$B_{n},C_{n},D_{n}$ and $\overline{\mathfrak{g}}$ its Levi subalgebra of type
$A_{n-1}$ obtained by removing the node $n$ in the Dynkin diagram, the
positivity of the polynomials $K_{\nu,\mu}(1,q)$ was recently conjectured by
Lee in \cite{lee}.
\end{Rem}

\section{The double deformation $K_{\nu,\mu}(p+1,q+1)$}

Our goal in this section is to prove that the polynomials $K_{\nu,\mu
}(p+1,q+1)$ have nonnegative coefficients for any dominant weight $\lambda$
and any weight $\mu$. This follows by adapting arguments due to Panyushev in
\cite{Pan} for the classical one-parameter Lusztig $q$-analogue.\ Set
\begin{multline*}
\Delta=\frac{\prod_{\alpha\in R_{+}}\left(  1-e^{-\alpha}\right)  }%
{\prod_{\alpha\in R_{+}\setminus\overline{R}_{+}}\left(  1-(p+1)e^{-\alpha
}\right)  \prod_{\alpha\in\overline{R}_{+}}\left(  1-(q+1)e^{-\alpha}\right)
}=\\
\frac{e^{-\rho}a_{\rho}}{\prod_{\alpha\in R_{+}\setminus\overline{R}_{+}%
}\left(  1-(p+1)e^{-\alpha}\right)  \prod_{\alpha\in\overline{R}_{+}}\left(
1-(q+1)e^{-\alpha}\right)  }%
\end{multline*}

\begin{Th}
\ \label{Th_dooublePan}

\begin{enumerate}
\item We have $\Delta=\sum_{\mu\in Q_{+}}K_{0,-\mu}(p+1,q+1)e^{-\mu}$ where
each polynomial $K_{0,-\mu}(p+1,q+1)$ belongs to $\mathbb{N}[p,q]$.

\item For $\nu\in P_{+}$ and $\mu\in P$, the polynomial $K_{\nu,\mu}%
(p+1,q+1)$, regarded as a polynomial in $\mathbb{Z}[p,q]$, has nonnegative
coefficients. Moreover, we have
\[
\sum_{\mu\in P}K_{\nu,\mu}(p+1,q+1)e^{\mu}=\Delta s_{\nu}%
\]
and the decomposition%
\begin{equation}
K_{\nu,\mu}(p+1,q+1)=\sum_{\beta\in Q_{+}}K_{\nu,\mu+\beta}K_{0,-\beta
}(p+1,q+1)\in\mathbb{N}[p,q]. \label{dec(q+1)}%
\end{equation}

\end{enumerate}
\end{Th}

\begin{proof}
We resume the notation of Section \ref{Section_DD}. First, there should exist
some polynomials $m_{\mu}(p+1,q+1)\in\mathbb{Z}[p+1,q+1],\mu\in Q_{+}$ such
that
\begin{equation}
\Delta=\frac{\prod_{\alpha\in R_{+}\setminus\overline{R}_{+}}\left(
1-e^{-\alpha}\right)  }{\prod_{\alpha\in R_{+}\setminus\overline{R}_{+}%
}\left(  1-(p+1)e^{-\alpha}\right)  }\times\frac{\prod_{\alpha\in\overline
{R}_{+}}\left(  1-e^{-\alpha}\right)  }{\prod_{\alpha\in\overline{R}_{+}%
}\left(  1-(q+1)e^{-\alpha}\right)  }=\sum_{\mu\in Q_{+}}m_{\mu}%
(p+1,q+1)e^{-\mu}. \label{Delta}%
\end{equation}

Recall that we have%
\begin{multline*}
\prod_{\alpha\in R_{+}}\left(  1-e^{-\alpha}\right)  =\sum_{w\in W}%
\varepsilon(w)e^{w(\rho)-\rho}\text{ and }\\
\prod_{\alpha\in R_{+}\setminus\overline{R}_{+}}\frac{1}{1-(p+1)e^{-\alpha}%
}\prod_{\alpha\in\overline{R}_{+}}\frac{1}{1-(q+1)e^{-\alpha}}=\sum_{\beta\in
Q_{+}}\mathrm{P}_{p+1,q+1}(\beta)e^{-\beta}.
\end{multline*}

This gives
\begin{multline*}
\Delta=\sum_{\beta\in Q_{+}}\sum_{w\in W}\varepsilon(w)\mathrm{P}%
_{p+1,q+1}(\beta)e^{w(\rho)-\rho-\beta}=\\
\sum_{\mu\in Q_{+}}\sum_{w\in W}\varepsilon(w)\mathrm{P}_{p+1,q+1}%
(w(\rho)-\rho+\mu)e^{-\mu}=\sum_{\mu\in Q_{+}}K_{0,-\mu}(p+1,q+1)e^{-\mu}%
\end{multline*}
by setting $\mu=\beta+\rho-w(\rho)$ which permits to conclude that $m_{\mu
}(p+1,q+1)=K_{0,-\mu}(p+1,q+1)$ for any $\mu\in Q_{+}$. We also have for any
positive root $\alpha$%
\begin{equation}
\frac{1-e^{-\alpha}}{1-(p+1)e^{-\alpha}}=1+\sum_{k\geq1}p(p+1)^{k-1}%
e^{-k\alpha} \label{MagicSeries}%
\end{equation}
and a similar identity by replacing $p$ by $q$. Therefore the formal expansion
of $\Delta$ should have nonnegative integer coefficients as claimed.

Now introduce the generating series
\[
u_{\nu}:=\sum_{\mu\in P}K_{\nu,\mu}(p+1,q+1)e^{\mu}.
\]

We have
\begin{multline*}
u_{\nu}=\sum_{\mu\in P}\sum_{w\in W}\varepsilon(w)\mathrm{P}_{p+1,q+1}%
(w(\nu+\rho)-\mu-\rho)e^{\mu}=\sum_{w\in W}\varepsilon(w)\sum_{\gamma\in
P}\mathrm{P}_{p+1,q+1}(\gamma)e^{w(\nu+\rho)-\gamma-\rho}=\\
\sum_{w\in W}\varepsilon(w)\left(  \sum_{\gamma\in P}\mathrm{P}_{p+1,q+1}%
(\gamma)e^{-\gamma}\right)  e^{w(\nu+\rho)-\rho}=\frac{e^{-\rho}\sum_{w\in
W}\varepsilon(w)e^{w(\nu+\rho)}}{\prod_{\alpha\in R_{+}\setminus\overline
{R}_{+}}\left(  1-(p+1)e^{-\alpha}\right)  \prod_{\alpha\in\overline{R}_{+}%
}\left(  1-(q+1)e^{-\alpha}\right)  }\\
=\Delta\frac{a_{\nu+\rho}}{a_{\rho}}=\Delta s_{\nu}%
\end{multline*}
where we set $\gamma=w(\nu+\rho)-\mu-\rho$ in the second equality. Now we get
by using Assertion 1%
\begin{subequations}
\begin{multline*}
u_{\nu}=\sum_{\mu\in P}K_{\nu,\mu}(p+1,q+1)e^{\mu}=\sum_{\beta\in Q_{+}%
}K_{0,-\beta}(p+1,q+1)e^{-\beta}\times\sum_{\gamma\in P}K_{\nu,\gamma
}e^{\gamma}=\\
\sum_{\beta,\gamma}K_{\nu,\gamma}K_{0,-\beta}(p+1,q+1)e^{\gamma-\beta}%
=\sum_{\mu\in P}\sum_{\beta\in Q_{+}}K_{\nu,\mu+\beta}K_{0,-\beta
}(p+1,q+1)e^{\mu}%
\end{multline*}
where we set $\mu=\gamma-\beta$ in the last equality. This gives
\end{subequations}
\begin{equation}
K_{\nu,\mu}(p+1,q+1)=\sum_{\beta\in Q_{+}}K_{\nu,\mu+\beta}K_{0,-\beta
}(p+1,q+1)\in\mathbb{N}[p,q] \label{final}%
\end{equation}
because the coefficients $K_{\nu,\mu+\beta}$ are nonnegative integers and we
have already proved that the polynomials $K_{0,-\beta}(p+1,q+1)$ belongs to
$\mathbb{N}[p,q]$.
\end{proof}

\begin{Rem}
Observe that in the decomposition (\ref{dec(q+1)}) the sum is finite since the
number of weights of the representation $V(\lambda)$ is and therefore
$K_{\nu,\mu+\beta}$ is nonzero only for a finite number of weights $\beta$.
\end{Rem}

\section{Stabilized version of the double deformation $K_{\nu,\mu}(p,q)$}

\subsection{The stabilization phenomenon}

We resume the notation of Section \ref{Section_DD}.\ In particular, we recall
the decomposition $\rho=\overline{\rho}+\rho^{\diamondsuit}$.\ Since
$\rho=\sum_{i\in I}\omega_{i}$ and $\overline{\rho}=\sum_{i\in\overline{I}%
}\omega_{i}$, we must have
\[
\rho^{\diamondsuit}=\sum_{i\in I\setminus\overline{I}}\omega_{i}.
\]
The goal of this section is to prove that, once the two dominant weights $\nu$
and $\mu$ are fixed, the sequence of polynomials
\[
K_{\nu+k\rho^{\diamondsuit},\mu+k\rho^{\diamondsuit}}(p,q),k\in\mathbb{N}%
\]
stabilizes and its limit has interesting positivity properties under a natural
hypothesis inspired by the classical Lie algebras setting of types
$B_{n},C_{n},D_{n}$ and their parabolic restrictions to type $A_{n-1}$
exploited in \cite{Le061} and \cite{Le06}. We will first need the following lemma.

\begin{Lem}
\label{Lem_stable}Consider a dominant weight $\nu$ in $P_{+}$ and a weight
$\mu\in P$ such that $\mu\leq\nu$. Then, there exists a nonnegative integer
$k_{0}$ such that
\[
w(\nu+\rho+k\rho^{\diamondsuit})-(\mu+\rho+k\rho^{\diamondsuit})\notin Q_{+}%
\]
for any $k\geq k_{0}$ and any $w\in W\setminus\overline{W}$.
\end{Lem}

\begin{proof}
Consider $w\in W\setminus\overline{W}$. Since $\rho^{\diamondsuit}$ belongs to
$P_{+}$ as a sum of fundamental weights, we must have $\rho^{\diamondsuit
}-w(\rho^{\diamondsuit})\in Q_{+}$, that is $\rho^{\diamondsuit}%
-w(\rho^{\diamondsuit})$ is a nonnegative sum of positive roots in $R_{+}$.
Moreover, $w(\rho^{\diamondsuit})\neq\rho^{\diamondsuit}$ because $w\in
W\setminus\overline{W}$ and the stabilizer of $\rho^{\diamondsuit}$ under the
action of $W$ is be $\overline{W}=\langle s_{i}\mid i\in\overline{I}\rangle
$.\ Therefore, we can set $\rho^{\diamondsuit}-w(\rho^{\diamondsuit}%
)=\beta_{w}$ with $\beta_{w}\in Q_{+}\setminus\{0\}$. Now, we get for any
nonnegative integer $k$%
\[
w(\nu+\rho+k\rho^{\diamondsuit})-(\mu+\rho+k\rho^{\diamondsuit})=w(\nu
+\rho)-(\nu+\rho)+(\nu-\mu)-k\beta_{w}.
\]
We have $w(\nu+\rho)-(\nu+\rho)<0$. Moreover, $\nu-\mu$ does not depends on
$k$ and $\beta_{w}>0$. There thus exists an integer $k_{w}$ such that
$w(\nu+\rho)-(\nu+\rho)+(\nu-\mu)-k\beta_{w}\notin Q_{+}$ for $k\geq k_{w}$
sufficiently large. We are done by considering $k_{0}=\mathrm{max}\{k_{w}\mid
w\in W\setminus\overline{W}\}$ (recall $W$ is finite).
\end{proof}

\bigskip

\begin{Prop}
\label{Prop-Stable}For any $\nu\in P_{+}$, any $\mu\in P$ and any
$\overline{\lambda}+\gamma\in\overline{P}_{+}$ such that $\overline{\lambda
}+\gamma\geq\mu$, there exists a nonnegative integer $k_{0}$ such that for any
$k\geq k_{0}$%
\[
b_{\nu+k\rho^{\Diamond},\overline{\lambda}+\gamma+k\rho^{\Diamond}%
}(p)=b_{\lambda+k_{0}\rho^{\Diamond},\overline{\lambda}+\gamma+k_{0}%
\rho^{\Diamond}}(p)=\sum_{\overline{w}\in\overline{W}}\varepsilon(\overline
{w})\mathrm{\hat{P}}_{p}(\overline{w}(\nu+\overline{\rho})-(\overline{\lambda
}+\gamma+\overline{\rho})).
\]
In particular, $b_{\nu+k\rho^{\Diamond},\overline{\lambda}+\gamma
+k\rho^{\Diamond}}(p)$ does not depend on $k$ for $k\geq k_{0}$ and we can set
in this case $b_{\nu,\overline{\lambda}+\gamma}^{\mathrm{stab}}(p):=b_{\nu
+k\rho^{\Diamond},\overline{\lambda}+\gamma+k\rho^{\Diamond}}(p)$.
\end{Prop}

\begin{proof}
By Lemma \ref{Lem_stable}, there exists a nonnegative integer $k_{0}$ such
that for any $k\geq k_{0}$ and any $w\in W\setminus\overline{W}$, we have
$w(\nu+\rho+k\rho^{\diamondsuit})-(\mu+\rho+k\rho^{\diamondsuit})\notin Q_{+}%
$. Since we then have $\overline{\lambda}+\gamma\geq\mu$ we also obtain that%
\[
w(\nu+\rho+k\rho^{\diamondsuit})-(\rho+\overline{\lambda}+\gamma
+k\rho^{\diamondsuit})=w(\nu+\rho+k\rho^{\diamondsuit})-(\mu+\rho
+k\rho^{\diamondsuit})-(\overline{\lambda}+\gamma-\mu)\notin Q_{+}%
\]
This gives%
\begin{multline*}
b_{\nu+k\rho^{\Diamond},\overline{\lambda}+\gamma+k\rho^{\Diamond}}%
(p)=\sum_{w\in W}\varepsilon(w)\mathrm{\hat{P}}_{p}(w(\nu+\overline{\rho
}+k\rho^{\Diamond})-(\overline{\lambda}+\gamma+\overline{\rho}+k\rho
^{\Diamond}))=\\
\sum_{\overline{w}\in\overline{W}}\varepsilon(\overline{w})\mathrm{\hat{P}%
}_{p}(\overline{w}(\nu+\overline{\rho}+k\rho^{\Diamond})-(\overline{\lambda
}+\gamma+\overline{\rho}+k\rho^{\Diamond}))=\sum_{\overline{w}\in\overline{W}%
}\varepsilon(\overline{w})\mathrm{\hat{P}}_{p}(\overline{w}(\nu+\overline
{\rho})-(\overline{\lambda}+\gamma+\overline{\rho}))
\end{multline*}
because $\mathrm{\hat{P}}_{p}(\beta)=0$ for $\beta\notin Q_{+}$ and
$\overline{w}(\rho^{\Diamond})=\rho^{\Diamond}$ for any $\overline{w}$ in
$\overline{W}$.
\end{proof}


\begin{Rem}
In \cite{WL2}, we establish a nonnegative expansion of
\[
\prod_{\alpha\in R_{+}\setminus\overline{R}_{+}}\frac{1}{1-pe^{\alpha}}%
\]
in $\overline{\mathfrak{g}}$-Weyl characters from which it becomes possible to
deduce that the stabilized polynomials $b_{\nu,\mu}^{\mathrm{stab}}(p)$ belong
to $\mathbb{N}[p]$ for any $\nu\in P_{+}$ and any $\mu=\mu^{\diamondsuit}%
\in\overline{P}_{+}$ (see Prop 3.9 in \cite{WL2}).
\end{Rem}

\begin{Th}
\label{Th_stable}For any dominant weight $\nu\in P_{+}$ and any weight $\mu$
in $P$ there exists a nonnegative integer $k_{0}$ such that, for any $k\geq
k_{0}$ we have the decomposition%
\[
K_{\nu+k\rho^{\diamondsuit},\mu+k\rho^{\diamondsuit}}(p,q)=\sum_{\overline
{\lambda}\in\overline{P}_{+}}b_{\nu,\overline{\lambda}+\mu^{\diamondsuit}%
}^{\mathrm{stab}}(p)\overline{K}_{\overline{\lambda},\overline{\mu}}(q).
\]
In particular, the stabilized double deformation $K_{\nu+k\rho^{\diamondsuit
},\mu+k\rho^{\diamondsuit}}(p,q)=K_{\nu,\mu}^{\mathrm{stab}}(p,q)$ does not
depend on $k\geq k_{0}$.
\end{Th}

\begin{proof}
By using Lemma \ref{Lem_stable} and the definition (\ref{Def_Double}) of
$K_{\nu+k\rho^{\diamondsuit},\mu+k\rho^{\diamondsuit}}(p,q)$, we get%
\[
K_{\nu+k\rho^{\diamondsuit},\mu+k\rho^{\diamondsuit}}(p,q)=\sum_{\overline
{w}\in\overline{W}}\varepsilon(\overline{w})\mathrm{P}_{p,q}(\overline{w}%
(\nu+\rho)-\mu-\rho).
\]
Indeed, by using the integer $k_{0}$ of Lemma \ref{Lem_stable}, for $k\geq
k_{0}$, we have for any $w\in W\setminus\overline{W}$%
\[
w(\nu+\rho+k\rho^{\diamondsuit})-\mu-\rho-k\rho^{\diamondsuit}\notin Q_{+}%
\]
and then $\mathrm{P}_{p,q}(\overline{w}(\nu+\rho+k\rho^{\diamondsuit}%
)-\mu-\rho-k\rho^{\diamondsuit})=0$. In the remaining cases
\[
\overline{w}(\nu+\rho+k\rho^{\diamondsuit})-\mu-\rho-k\rho^{\diamondsuit
}=\overline{w}(\nu+\rho)-\mu-\rho
\]
since $\overline{w}(\rho^{\diamondsuit})=\rho^{\diamondsuit}$. Now, we can
argue exactly as in the proof of Theorem \ref{Th_Dec_WLA} (but by replacing
the set $U$ by $\{id\}$). The equality (\ref{fund}) is then replaced by
\[
K_{\nu,\mu}(p,q)=\sum_{\overline{\lambda}\in\overline{P}_{+}}\left(
\sum_{\overline{w}\in\overline{W}}\varepsilon(\overline{w})\mathrm{\hat{P}%
}_{p}\left(  \overline{w}(\nu+\rho)-\mu^{\diamondsuit}-\overline{\lambda}%
-\rho)\right)  \right)  \overline{K}_{\overline{\lambda},\overline{\mu}%
}(q)=\sum_{\overline{\lambda}\in\overline{P}_{+}}b_{\nu,\overline{\lambda}%
+\mu^{\diamondsuit}}^{\mathrm{stab}}(p)\overline{K}_{\overline{\lambda
},\overline{\mu}}(q)
\]
where the last equality follows from Proposition \ref{Prop-Stable} with
$\gamma=\mu^{\diamondsuit}$ and $\overline{\lambda}+\mu^{\diamondsuit}\geq\mu$
since $\overline{\lambda}\geq\overline{\mu}$ when $\overline{K}_{\overline
{\lambda},\overline{\mu}}(q)\neq0$.
\end{proof}

\begin{Rem}
\label{Rem_delta}In Lemma \ref{Lem_stable}, Proposition \ref{Prop-Stable} and
Theorem \ref{Th_stable}, the crucial argument is the invariance of
$\rho^{\diamondsuit}$ under the action of $\overline{W}$.\ Instead of choosing
first the parabolic root system and then the direction $\rho^{\diamondsuit}$
in the Weyl chamber of $\mathfrak{g}$, it is also interesting to first
consider a weight $\delta$ in $P_{+}$ and next the set $\overline{I}:=\{i\in
I\mid\langle\delta,\alpha_{i}\rangle=0\}$. Then $\overline{W}$ stabilizes
$\delta$. By using the same arguments, we can then establish the theorem
below, companion of Theorem \ref{Th_stable} where the decomposition of
$\delta$ on the basis of fundamental weights can make appear multiplicities
greater than one.
\end{Rem}

\begin{Th}
\label{Th_dela_stable}Assume that $\overline{I}$ is chosen as previously from
the dominant weight $\delta$. Then for any weight $\mu$ in $P$ there exists a
nonnegative integer $k_{0}$ such that, for any $k\geq k_{0}$ we have the
decomposition%
\[
K_{\nu+k\delta,\mu+k\delta}(p,q)=\sum_{\overline{\lambda}\in\overline{P}_{+}%
}b_{\nu,\overline{\lambda}+\mu^{\diamondsuit}}^{\delta\text{-}\mathrm{stab}%
}(p)\overline{K}_{\overline{\lambda},\overline{\mu}}(q)
\]
where $b_{\nu,\overline{\lambda}+\gamma}^{\delta\text{-}\mathrm{stab}%
}(p):=b_{\nu+k\delta,\overline{\lambda}+\gamma+k\delta}(p)$. In particular,
the stabilized double deformation $K_{\nu,\mu}^{\delta\text{-}\mathrm{stab}%
}(p,q)=K_{\nu+k\delta,\mu+k\delta}(p,q)$ does not depend on $k\geq k_{0}$.
\end{Th}

\subsection{Positivity of the stabilized forms}

In this paragraph, we shall need the following additional hypothesis.

\begin{Hyp}
\label{Hypo_H}We assume in this section that $\langle\rho^{\diamondsuit
},\alpha\rangle=c$ is a fixed positive integer for any $\alpha\in
R_{+}\setminus\overline{R}_{+}$\ which does not depend on the choice of
$\alpha$ in $R_{+}\setminus\overline{R}_{+}$.
\end{Hyp}

\begin{Rem}
\ 

\begin{enumerate}
\item Observe that our condition is in particular satisfied for the classical
root system of type $C_{n}$ or $D_{n}$ when $\overline{I}=\{1,\ldots,n-1\}$ is
the set of labels of the parabolic subsystem of type $A_{n-1}$. Then for any
$\alpha\in R_{+}\setminus\overline{R}_{+}$, we have $\langle\rho
^{\diamondsuit},\alpha\rangle=\langle\omega_{n},\alpha\rangle=2$ in type
$C_{n}$ and $\langle\rho^{\diamondsuit},\alpha\rangle=\langle\omega_{n}%
,\alpha\rangle=1$ in type $D_{n}$.

\item It is also satisfied for the parabolic subsystems of type $B_{n-1}$ or
$D_{n-1}$ inside the root systems of type $B_{n}$ or $D_{n}$, respectively. We
then indeed get $\langle\rho^{\diamondsuit},\alpha\rangle=\langle\omega
_{1},\alpha\rangle=1$ for any $\alpha\in R_{+}\setminus\overline{R}_{+}$.

\item Another case where the hypothesis is satisfied if for $\mathfrak{g}$ of
type $A_{n-1}$ and $\overline{\mathfrak{g}}$ the Levi subalgebra obtained by
removing the $k$-th node in the Dynkin diagram of $\mathfrak{g}$. In this
case, we obtain $\langle\rho^{\diamondsuit},\alpha\rangle=\langle\omega
_{k},\alpha\rangle=1$ for any $\alpha\in R_{+}\setminus\overline{R}_{+}$.

\item As in \cite{LOS}, it could be also interesting to extend the
stabilization phenomenon to deformations of the Lusztig analogues defined by
using weight functions on $p$ and $q$ depending on the length of the root
considered. When $p=q$, this is in particular unavoidable for equating the
one-dimensional sums for affine root systems with $q$-analogues of weight
multiplicities. For simplicity, we do not pursue in this direction here.
\end{enumerate}
\end{Rem}

Consider $\gamma$ in $\hat{Q}_{+}$ and a decomposition $\gamma=\sum_{\alpha\in
R_{+}\setminus\overline{R}_{+}}a_{\alpha}\alpha$ where the $a_{\alpha}$'s are
nonnegative integers. Then
\[
L(\gamma):=\frac{1}{c}\langle\rho^{\diamondsuit},\gamma\rangle=\frac{1}{c}%
\sum_{\alpha\in R_{+}\setminus\overline{R}_{+}}a_{\alpha}\langle
\rho^{\diamondsuit},\alpha\rangle=\sum_{\alpha\in R_{+}\setminus\overline
{R}_{+}}a_{\alpha}%
\]
by Hypothesis \ref{Hypo_H}. This shows that the number of positive roots in
$R_{+}\setminus\overline{R}_{+}$ appearing in a relevant decomposition of
$\gamma$ is always equal to $L(\gamma)$. Therefore, we have $\mathrm{\hat{P}%
}_{p}(\gamma)=p^{L(\gamma)}\mathrm{\hat{P}}_{1}(\gamma)$ for any $\gamma$ in
$\hat{Q}_{+}$. Moreover, we have by Proposition \ref{Prop-Stable}%
\[
b_{\nu,\overline{\lambda}+\mu^{\diamondsuit}}^{\mathrm{stab}}(q)=\sum
_{\overline{w}\in\overline{W}}\varepsilon(\overline{w})\mathrm{\hat{P}}%
_{p}(\overline{w}(\nu+\overline{\rho})-(\overline{\lambda}++\mu^{\diamondsuit
}+\overline{\rho})).
\]
For any $\overline{w}\in\overline{W}$, and with $\gamma=\overline{w}%
(\nu+\overline{\rho})-(\overline{\lambda}+\mu^{\diamondsuit}+\overline{\rho})$%
\begin{multline*}
L(\gamma)=\frac{1}{c}\langle\rho^{\diamondsuit},\overline{w}(\nu
+\overline{\rho})-(\overline{\lambda}+\mu^{\diamondsuit}+\overline{\rho
})\rangle=\frac{1}{c}\langle\rho^{\diamondsuit},\overline{w}(\nu
+\overline{\rho})\rangle-\frac{1}{c}\langle\rho^{\diamondsuit},\overline
{\lambda}+\mu^{\diamondsuit}+\overline{\rho}\rangle=\\
\frac{1}{c}\langle\overline{w}^{-1}(\rho^{\diamondsuit}),\nu+\overline{\rho
}\rangle-\frac{1}{c}\langle\rho^{\diamondsuit},\overline{\lambda}%
+\mu^{\diamondsuit}+\overline{\rho}\rangle=\frac{1}{c}\langle\rho
^{\diamondsuit},\nu+\overline{\rho}\rangle-\frac{1}{c}\langle\rho
^{\diamondsuit},\overline{\lambda}+\mu^{\diamondsuit}+\overline{\rho}%
\rangle=\\
\frac{1}{c}\langle\rho^{\diamondsuit},\nu-\overline{\lambda}-\mu
^{\diamondsuit}\rangle=L(\nu-\overline{\lambda}-\mu^{\diamondsuit})
\end{multline*}
because $\rho^{\diamondsuit}$ is stabilized by $\overline{W}$. Therefore%
\[
b_{\nu,\overline{\lambda}}^{\mathrm{stab}}(p)=p^{L(\nu-\overline{\lambda}%
-\mu^{\diamondsuit})}b_{\nu,\overline{\lambda}}^{\mathrm{stab}}(1)\text{ for
any }\nu\in P_{+}\text{ and any }\overline{\lambda}\in\overline{P}_{+}.
\]
This proves that $b_{\nu,\overline{\lambda}}^{\mathrm{stab}}(q)$ has
nonnegative integer coefficients. We then get the following Corollary of
Theorem \ref{Th_stable}.

\begin{Cor}
Under Hypothesis \ref{Hypo_H}, for any dominant weights $\nu\in P_{+}$ and
$\mu\in\overline{P}_{+}$, the polynomial $K_{\nu,\mu}^{\mathrm{stab}}(p,q)$
belongs to $\mathbb{N}[p,q]$. Moreover, we have%
\[
K_{\nu+k\rho^{\diamondsuit},\mu+k\rho^{\diamondsuit}}(p,q)=\sum_{\overline
{\lambda}\in\overline{P}_{+}}p^{L(\nu-\overline{\lambda}-\mu^{\diamondsuit}%
)}b_{\nu,\overline{\lambda}+\mu^{\diamondsuit}}^{\mathrm{stab}}(1)\overline
{K}_{\overline{\lambda},\overline{\mu}}(q).
\]

\end{Cor}

\begin{Rem}
Here again, one can follow the same line as in Remark \ref{Rem_delta} and
consider a dominant weight $\delta$ in $P_{+}$ with $\overline{I}:=\{i\in
I\mid\langle\delta,\alpha_{i}\rangle=0\}$. Then replace Hypothesis
\ref{Hypo_H} by the condition
\begin{equation}
\langle\delta,\alpha\rangle=c \label{H2}%
\end{equation}
is a positive constant for any $\alpha\in R_{+}\setminus\overline{R}_{+}$ and
set $L_{\delta}(\gamma):=\frac{1}{c}\langle\delta,\gamma\rangle$ for any
$\gamma=\sum_{\alpha\in R_{+}\setminus\overline{R}_{+}}a_{\alpha}\alpha$. This
gives the following companion corollary of Theorem \ref{Th_dela_stable}.
\end{Rem}

\begin{Cor}
Assume $\delta$ is fixed in $P_{+}$ and then the convention and hypotheses of
the previous remark.\ Then, for any dominant weights $\nu\in P_{+}$ and
$\mu\in\overline{P}_{+}$, the polynomial $K_{\nu,\mu}^{\delta\text{-}%
\mathrm{stab}}(p,q)$ belongs to $\mathbb{N}[p,q]$. Moreover, we have%
\[
K_{\nu,\mu}^{\delta\text{-}\mathrm{stab}}(p,q)=K_{\nu+k\delta,\mu+k\delta
}(p,q)=\sum_{\overline{\lambda}\in\overline{P}_{+}}p^{L_{\delta}(\nu
-\overline{\lambda}-\mu^{\diamondsuit})}b_{\nu,\overline{\lambda}%
+\mu^{\diamondsuit}}^{\delta\text{-}\mathrm{stab}}(1)\overline{K}%
_{\overline{\lambda},\overline{\mu}}(q).
\]

\end{Cor}

\section{Combinatorial descriptions}

\subsection{Combinatorial description of the double deformation $K_{\nu,\mu
}(p,q)$}

A difficult question in the theory of Lusztig $q$-analogues of weight
multiplicities consists in their combinatorial description. In type $A_{n-1}$
this was achieved by Lascoux and Sch\"{u}tzenberger in \cite{LSc1}.\ For the
classical types, the description is only partially known. We refer the
interested reader to \cite{CKL}, \cite{LN}, \cite{lee} and the references
therein.\ In view to Theorem \ref{Th_Dec_WLA}, the combinatorial description
of the polynomials $K_{\nu,\mu}(p,q)$ reduces to the combinatorial description
of the polynomials $b_{\nu,\overline{\lambda}+\mu^{\diamondsuit}}(p)$ and
$\overline{K}_{\overline{\lambda},\overline{\mu}}(q)$.

For $p=1$, the branching coefficient $b_{\nu,\overline{\lambda}+\mu
^{\diamondsuit}}(1)$ can be easily described thanks to crystal graph theory or
Littelmann path theory since both are compatible with parabolic
restrictions.\ Recall that to each dominant weight $\nu$, is associated a
crystal $B(\nu)$ which can be regarded as the combinatorial skeleton of the
representation $V(\nu).$ This is an oriented graph where the arrows are
described in terms of the crystal operators $\tilde{f}_{i}$ and $\tilde{e}%
_{i}$ with $i\in I$. Each vertex $b$ in $B(\nu)$ has a weight $\mathrm{wt}%
(b)\in P$ so that%
\begin{equation}
s_{\nu}=\sum_{b\in B(\nu)}e^{\mathrm{wt}(b)}. \label{CharacterCrystals}%
\end{equation}
In particular the cardinality of $B(\nu)$ is equal to the dimension of
$V(\nu)$. Also $B(\nu)$ has a unique source vertex $b_{\lambda}$ which is the
unique vertex killed by all the crystal operators $\tilde{e}_{i},i\in I$. The
multiplicity $b_{\nu,\overline{\lambda}+\mu^{\diamondsuit}}(1)$ is then equal
to the cardinality of the set
\[
H_{\overline{\lambda}+\mu^{\diamondsuit}}=\{b\in B(\nu)\mid\tilde{e}%
_{i}(b)=0,i\in\overline{I}\text{ and }\mathrm{wt}(b)=\overline{\lambda}%
+\mu^{\diamondsuit}\}.
\]
This thus reduces the combinatorial description of the polynomials $K_{\nu
,\mu}(1,q)$ to that of the polynomials $\overline{K}_{\overline{\lambda
},\overline{\mu}}(q)$, i.e. to that of the ordinary Lusztig $q$-analogues.
More precisely, denote by $\overline{B}(b)$ the parabolic crystal with highest
weight vertex $b\in H_{\overline{\lambda}+\mu^{\diamondsuit}}$ (that is, the
crystal $\overline{B}(b)$ is obtained by applying the crystal operators
$\tilde{f}_{i},i\in\overline{I}$ to $b$). Then $\mathrm{wt}(b)$ belongs to
$\overline{P}_{+}$. Assume one knows a map%
\begin{equation}
\overline{\mathrm{ch}}:\overline{B}(b)\rightarrow\mathbb{N} \label{Ch_bar}%
\end{equation}
such that
\[
\sum_{b^{\prime}\in\overline{B}(b)_{\mu}}q^{\overline{\mathrm{ch}}(b^{\prime
})}=\overline{K}_{\overline{\lambda}+\mu^{\diamond},\mu}(q)=\overline
{K}_{\overline{\lambda},\overline{\mu}}(q)
\]
where $\mathrm{wt}(b)=\overline{\lambda}+\mu^{\diamondsuit}$ belongs to
$\overline{P}_{+}$ and $\overline{B}(b)_{\mu}=\{b^{\prime}\in\overline
{B}(b)\mid\mathrm{wt}(b^{\prime})=\mu\}$.\ Then, we have by Theorem
\ref{Th_Dec_WLA}%
\[
K_{\nu,\mu}(1,q)=\sum_{b\in H_{\overline{\lambda}+\mu^{\diamondsuit}}}%
\sum_{b^{\prime}\in\overline{B}(b)_{\mu}}q^{\overline{\mathrm{ch}}(b^{\prime
})}.
\]
In particular, when $\mathfrak{g}$ is of classical type $B_{n},C_{n}$ or
$D_{n}$ and $\overline{\mathfrak{g}}$ is isomorphic to $\mathfrak{gl}_{n}$
(i.e. $\overline{I}=I\setminus\{n\}$), there are various combinatorial models
for the crystal $B(\nu)$ and also for computing the branching coefficients
$b_{\nu,\overline{\lambda}+\mu^{\diamondsuit}}(1)$ without constructing the
whole crystal $B(\nu)$ (see \cite{Kwon} and the reference therein). Moreover
the parabolic Lusztig $q$-analogues $\overline{K}_{\overline{\lambda
},\overline{\mu}}(q)$ then coincide with the Kostka polynomials of type
$A_{n-1}$ whose combinatorial description was given by Lascoux and
Sch\"{u}tzenberger in \cite{LSc1} thanks to the so-called charge statistics on
semistandard tableaux (see also \cite{Pat} for a recent geometric description
of the charge). This charge statistics can also be defined purely in terms of
crystals and then be used as a map $\overline{\mathrm{ch}}$ (see
(\ref{Ch_bar})). For the exceptional types, there is always a unique node in
$I$ to remove in order to get a maximal parabolic root system of type $A$.
Therefore, this strategy also works even if the combinatorial description of
the crystals is then more complicated.

\subsection{Combinatorial description of the double deformation $K_{\nu,\mu
}(p+1,q+1)$}

We will now explain how it is possible to get a combinatorial description of
the polynomials $K_{\nu,\mu}(p+1,q+1)$ when $\nu\in P_{+}$ and $\mu\in P$ in
terms of a polynomial statistics on the vertices of the crystal graph $B(\nu)$
thanks to Theorem \ref{Th_dooublePan}. To do this, let us consider two copies
$R_{+}^{b}$ and $R_{+}^{r}$ of the set of positive roots $R_{+}$ that can be
thought as positive roots colored in blue and red.\ Let $\mathcal{R}_{+}%
=R_{+}^{b}\sqcup R_{+}^{r}.$ We will use bold symbols to denote a colored
positive root $\boldsymbol{\alpha}$ in $\mathcal{R}_{+}$ and write $\left\vert
\boldsymbol{\alpha}\right\vert $ for its corresponding uncolored root in
$R_{+}$ (obtained by forgetting its color).\ For any multiset $S$ of
$\mathcal{R}_{+}$ (this means that a colored positive root may appear more
than once in $S$), denote by $d(S)$ and $e(S)$ respectively the numbers of
blue roots in $R_{+}^{b}\setminus\overline{R}_{+}^{b}$ and $\overline{R}%
_{+}^{b}$ it contains. Write for short $S\Subset\mathcal{R}_{+}$ when $S$ is a
multiset of $\mathcal{R}_{+}$ and for such a multiset, put $\Sigma
(S)=\sum_{\alpha\in S}\left\vert \boldsymbol{\alpha}\right\vert $ the sum of
the roots it contains once the colors are forgotten.

\begin{Lem}
\label{Lem_N(p,q)}We have
\begin{equation}
\frac{1}{\prod_{\alpha\in R_{+}\setminus\overline{R}_{+}}\left(
1-(p+1)e^{-\alpha}\right)  \prod_{\alpha\in\overline{R}_{+}}\left(
1-(q+1)e^{-\alpha}\right)  }=\sum_{\beta\in Q_{+}}N_{p,q}(\beta)e^{-\beta}
\label{RelaN}%
\end{equation}
where
\begin{equation}
N_{p,q}(\beta)=\sum_{S\Subset\mathcal{R}_{+}\mid\Sigma(S)=\beta}%
p^{d(S)}q^{e(S)}. \label{N_(p,q)}%
\end{equation}

\end{Lem}

\begin{proof}
We have for $\alpha\in R_{+}\setminus\overline{R}_{+}$%
\[
\frac{1}{1-(p+1)e^{-\alpha}}=\sum_{k\geq0}(p+1)^{k}e^{-k\alpha}=\sum_{k\geq
0}\sum_{a=0}^{k}\binom{k}{a}p^{a}e^{-k\alpha}=\sum_{S_{\alpha}\Subset
\{\boldsymbol{\alpha}^{b},\boldsymbol{\alpha}^{r}\}}p^{d(S_{\alpha}%
)}e^{-\Sigma(S_{\alpha})}%
\]
where the last sum runs over the multisets $S_{\alpha}$ containing only the
colored roots $\boldsymbol{\alpha}^{b}$ or $\boldsymbol{\alpha}^{r}$.
Similarly for $\alpha\in\overline{R}_{+}$, one gets%
\[
\frac{1}{1-(q+1)e^{-\alpha}}=\sum_{S_{\alpha}\Subset\{\boldsymbol{\alpha}%
^{b},\boldsymbol{\alpha}^{r}\}}q^{e(S_{\alpha})}e^{-\Sigma(S_{\alpha})}.
\]
The lemma follows by using that the left hand side of (\ref{RelaN}) is product
rational geometric series of the two previous types.
\end{proof}

\bigskip

Now recall we proved in Theorem \ref{Th_dooublePan} the identity%
\[
\Delta=\frac{\prod_{\alpha\in R_{+}}\left(  1-e^{-\alpha}\right)  }%
{\prod_{\alpha\in R_{+}\setminus\overline{R}_{+}}\left(  1-(p+1)e^{-\alpha
}\right)  \prod_{\alpha\in\overline{R}_{+}}\left(  1-(q+1)e^{-\alpha}\right)
}=\sum_{\beta\in Q_{+}}K_{0,-\beta}(p+1,q+1)e^{-\beta}.
\]
In order to get an analogue of Lemma \ref{Lem_N(p,q)} for $\Delta$, we need to
restrict the possible multisets of $\mathcal{R}_{+}$ that we consider. Given a
colored positive root $\boldsymbol{\alpha}$ and $S\Subset\mathcal{R}$, denote
by $m_{\boldsymbol{\alpha}}(S)$ its multiplicity in $S$. Let us say that
$S\Subset\mathcal{R}_{+}$ is admissible when for any red root
$\boldsymbol{\alpha}\in R_{+}^{b}$ such that $m_{\boldsymbol{\alpha}}(S)>0$,
its associated blue root $\boldsymbol{\alpha}^{\prime}$ such that $\left\vert
\boldsymbol{\alpha}\right\vert =\left\vert \boldsymbol{\alpha}^{\prime
}\right\vert $ also satisfies $m_{\boldsymbol{\alpha}^{\prime}}(S)>0$ (but we
do not assume here that $m_{\boldsymbol{\alpha}}(S)=m_{\boldsymbol{\alpha
}^{\prime}}(S)$). Roughly speaking, this means that $S$ is admissible when it
cannot contain a red positive root $\boldsymbol{\alpha}$ such that
$\alpha=\left\vert \boldsymbol{\alpha}\right\vert $ without also containing
the blue version of $\alpha$.\ Observe that the empty set is admissible. Write
for short $S\Subset_{a}\mathcal{R}_{+}$ when $S$ is an admissible multiset of
$\mathcal{R}_{+}$. Finally for any $\beta$ in $Q$ set%
\begin{equation}
R_{p,q}(\beta)=\sum_{S\Subset_{a}\mathcal{R}_{+}\mid\Sigma(S)=\beta}%
p^{d(S)}q^{e(S)} \label{R_(p,q)}%
\end{equation}
when $\beta\in Q_{+}$ and $R_{p,q}(\beta)=0$ otherwise.\ We get the following equality.

\begin{Lem}
\label{Lem_R(p,q)}For ant $\beta$ in $Q_{+}$ we have $K_{0,-\beta
}(p+1,q+1)=R_{p,q}(\beta).$
\end{Lem}

\begin{proof}
The proof follows the same line as Proof of Lemma \ref{N_(p,q)}.\ As in
(\ref{MagicSeries}), we get for $\alpha\in R_{+}\setminus\overline{R}_{+}$%
\begin{multline*}
\frac{1-e^{-\alpha}}{1-(p+1)e^{-\alpha}}=1+pe^{-\alpha}\sum_{k\geq0}%
(p+1)^{k}e^{-k\alpha}=1+pe^{-\alpha}\sum_{k\geq0}\sum_{a=0}^{k}\binom{k}%
{a}p^{a}e^{-k\alpha}=\\
\sum_{S_{\alpha}\Subset_{a}\{\boldsymbol{\alpha}^{b},\boldsymbol{\alpha}%
^{r}\}}p^{d(S_{\alpha})}e^{-\Sigma(S_{\alpha})}%
\end{multline*}
where the last sum runs over the restricted multisets $S_{\alpha}$ containing
only the colored roots $\boldsymbol{\alpha}^{b}$ or $\boldsymbol{\alpha}^{r}$.
Similarly for $\alpha\in\overline{R}_{+}$, one gets%
\[
\frac{1-e^{-\alpha}}{1-(q+1)e^{-\alpha}}=\sum_{S_{\alpha}\Subset
_{r}\{\boldsymbol{\alpha}^{b},\boldsymbol{\alpha}^{r}\}}q^{e(S_{\alpha}%
)}e^{-\Sigma(S_{\alpha})}.
\]
The lemma follows by using that $\Delta$ is a product of rational series of
the two previous types.
\end{proof}

\bigskip

We now define the polynomial $\mu$-charge statistics $\chi_{p,q}^{\mu}$ on the
vertices of $B(\lambda)$ as follows%
\[
\chi_{p,q}^{\mu}:\left\{
\begin{array}
[c]{c}%
B(\lambda)\rightarrow\mathbb{Z}_{\geq0}[p,q]\\
b\longmapsto R_{p,q}(\mathrm{wt}(b)-\mu)
\end{array}
\right.
\]
where the polynomials $R_{p,q}$ are defined in (\ref{R_(p,q)}) and the weight
function $\mathrm{wt}$ is the same as in (\ref{CharacterCrystals}). In
particular, we have $\chi_{p,q}^{\mu}(b)=0$ when $\mathrm{wt}(b)$ is not
greater or equal to $\mu$ in the dominant order.\ The following theorem
describes the polynomial $K_{\nu,\mu}(p+1,q+1)$ as the generating series for
the polynomial $\mu$-charge statistics $\chi_{p,q}$ over the vertices of the
crystal $B(\nu)$.

\begin{Th}
For any dominant weight $\nu$ and any weight $\mu$, we have%
\[
K_{\nu,\mu}(p+1,q+1)=\sum_{b\in B(\nu)_{\geq\mu}}\chi_{p,q}^{\mu}(b)
\]
where $B(\nu)_{\geq\mu}=\{b\in B(\nu)\mid\mathrm{wt}(b)\geq\mu\}$.
\end{Th}

\begin{proof}
By Theorem \ref{Th_dooublePan} and Lemma \ref{Lem_R(p,q)}, one obtains the
decomposition%
\[
K_{\nu,\mu}(p+1,q+1)=\sum_{\beta\in Q_{+}}K_{\nu,\mu+\beta}K_{0,-\beta
}(p+1,q+1)=\sum_{\beta\in Q_{+}}K_{\nu,\mu+\beta}R_{p,q}(\beta).
\]
Now, we have
\[
K_{\nu,\mu+\beta}=\mathrm{card}\left(  b\in B(\nu)\mid\mathrm{wt}(b)=\mu
+\beta\right)
\]
and therefore%
\begin{multline*}
K_{\nu,\mu}(p+1,q+1)=\sum_{\beta\in Q_{+}}\sum_{b\in B(\nu)\mid\mathrm{wt}%
(b)=\mu+\beta}R_{p,q}(\beta)=\sum_{b\in B(\nu)}R_{p,q}(\mathrm{wt}(b)-\mu)=\\
\sum_{b\in B(\nu)}\chi_{p,q}^{\mu}(b)=\sum_{b\in B(\nu)_{\geq\mu}}\chi
_{p,q}^{\mu}(b)
\end{multline*}
as desired since $\chi_{p,q}^{\mu}(b)=0$ when $\mathrm{wt}(b)$ is not greater
or equal to $\mu$.
\end{proof}


\begin{Exa}
We give below the tableau realization of the crystal graph $B(\lambda)$ of the
adjoint representation in type $A_{2}$ with highest weight $\lambda=\omega
_{1}+\omega_{2}$. We have $R_{+}=\{\alpha_{1},\alpha_{2},\alpha_{1}+\alpha
_{2})$. For $\mu=0$ we indicate the values of the polynomial statistics
$\chi_{p,q}^{\mu}(b)$ on the vertices $b$ such that $\mathrm{w}(b)\geq0$:%
\[%
\begin{tabular}
[c]{lllllll}
&  &  & $%
\begin{tabular}
[c]{|c|c}\cline{1-1}%
$2$ & \\\hline
$1$ & \multicolumn{1}{|c|}{$1$}\\\hline
\end{tabular}
\ $ &  &  & \\
&  & $%
\genfrac{.}{.}{0pt}{}{\text{{\tiny 1}}}{\swarrow}%
$ &  & $%
\genfrac{.}{.}{0pt}{}{\text{{\tiny 2}}}{\searrow}%
$ &  & \\
& $%
\begin{tabular}
[c]{|l|l}\cline{1-1}%
$2$ & \\\hline
$1$ & \multicolumn{1}{|l|}{$2$}\\\hline
\end{tabular}
\ $ &  &  &  & $%
\begin{tabular}
[c]{|l|l}\cline{1-1}%
$3$ & \\\hline
$1$ & \multicolumn{1}{|l|}{$1$}\\\hline
\end{tabular}
\ $ & \\
& \ \ {\tiny 2}$\downarrow$ &  &  &  & $\ \ ${\tiny 1}$\downarrow$ & \\
& $%
\begin{tabular}
[c]{|l|l}\cline{1-1}%
$2$ & \\\hline
$1$ & \multicolumn{1}{|l|}{$3$}\\\hline
\end{tabular}
\ $ &  &  &  & $%
\begin{tabular}
[c]{|l|l}\cline{1-1}%
$3$ & \\\hline
$1$ & \multicolumn{1}{|l|}{$2$}\\\hline
\end{tabular}
\ $ & \\
& \ \ {\tiny 2}$\downarrow$ &  &  &  & \ \ {\tiny 1}$\downarrow$ & \\
& $%
\begin{tabular}
[c]{|l|l}\cline{1-1}%
$3$ & \\\hline
$1$ & \multicolumn{1}{|l|}{$3$}\\\hline
\end{tabular}
\ $ &  &  &  & $%
\begin{tabular}
[c]{|l|l}\cline{1-1}%
$3$ & \\\hline
$2$ & \multicolumn{1}{|l|}{$2$}\\\hline
\end{tabular}
\ $ & \\
&  & $%
\genfrac{.}{.}{0pt}{}{\text{{\tiny 1}}}{\searrow}%
$ &  & $%
\genfrac{.}{.}{0pt}{}{\text{{\tiny 2}}}{\swarrow}%
$ &  & \\
&  &  & $%
\begin{tabular}
[c]{|l|l}\cline{1-1}%
$3$ & \\\hline
$2$ & \multicolumn{1}{|l|}{$3$}\\\hline
\end{tabular}
\ $ &  &  & \\
&  &  &  &  &  &
\end{tabular}
\]

\[%
\begin{tabular}
[c]{cccccc}%
$b\medskip$ & $%
\begin{tabular}
[c]{|l|l}\cline{1-1}%
$2$ & \\\hline
$1$ & \multicolumn{1}{|l|}{$3$}\\\hline
\end{tabular}
$ & $%
\begin{tabular}
[c]{|l|l}\cline{1-1}%
$3$ & \\\hline
$1$ & \multicolumn{1}{|l|}{$2$}\\\hline
\end{tabular}
$ & $%
\begin{tabular}
[c]{|l|l}\cline{1-1}%
$2$ & \\\hline
$1$ & \multicolumn{1}{|l|}{$2$}\\\hline
\end{tabular}
$ & $%
\begin{tabular}
[c]{|l|l}\cline{1-1}%
$3$ & \\\hline
$1$ & \multicolumn{1}{|l|}{$1$}\\\hline
\end{tabular}
$ & $%
\begin{tabular}
[c]{|c|c}\cline{1-1}%
$2$ & \\\hline
$1$ & \multicolumn{1}{|c|}{$1$}\\\hline
\end{tabular}
$\\
$\chi_{p,p}^{\mu}(b)$ & $1$ & $1$ & $q$ & $q$ & $q^{2}+q$%
\end{tabular}
\]

$\medskip$For example, when $b=%
\begin{tabular}
[c]{|c|c}\cline{1-1}%
$2$ & \\\hline
$1$ & \multicolumn{1}{|c|}{$1$}\\\hline
\end{tabular}
$ is the highest weight, the admissible multisets of colored roots summing up
$\lambda=\alpha_{1}+\alpha_{2}$ are%
\[
\{(\alpha_{1}+\alpha_{2})_{b}\}\text{ and }\{\{(\alpha_{1})_{b},\{\alpha
_{2})_{b}\}
\]
which gives $\chi_{p,p}^{\mu}(b)=q^{2}+q$. Finally we have%
\[
K_{\lambda,\mu}(q+1,q+1)=2+2q+(q^{2}+q)=q^{2}+3q+2=(q+1)^{2}+(q+1)
\]
as expected.
\end{Exa}

\subsection{Some analogues of the Hall-Littlewood polynomials}

Recall that the character ring $\mathbb{Z}[p,q]^{W}[P]$ of $\mathfrak{g}$ can
be endowed with a scalar product $\langle\cdot,\cdot\rangle_{(p,q)}$ such that
$\langle s_{\nu},s_{\mu}\rangle_{(p,q)}=\delta_{\nu,\mu}$ for any dominant
weights $\nu$ and $\mu$ in $P_{+}$.\ For any $\mu$ in $P_{+}$, one can then
define the polynomial%
\[
Q_{\mu}^{\prime}(p,q)=\frac{1}{a_{\rho}}\sum_{w\in W}\varepsilon(w)w\left(
\frac{e^{\mu+\rho}}{\prod_{\alpha\in R_{+}\setminus\overline{R}_{+}}\left(
1-pe^{\alpha}\right)  \prod_{\alpha\in\overline{R}_{+}}\left(  1-qe^{\alpha
}\right)  }\right)  .
\]

\begin{Prop}
We have
\[
Q_{\mu}^{\prime}(p,q)=\sum_{\nu\in P_{+},\mu\leq\nu}K_{\nu,\mu}(p,q)s_{\nu}.
\]
In particular $\{Q_{\mu}^{\prime}(p,q)\mid\mu\in P_{+}\}$ is a $\mathbb{Z}%
$-basis of $\mathbb{Z}[p,q]^{W}[P]$.
\end{Prop}

\begin{proof}
By using the expansion (\ref{(p,q)-Kostant}) of the $(p,q)$-Kostant partition,
we get%
\[
Q_{\mu}^{\prime}(p,q)=\sum_{\beta\in Q_{+}}\frac{1}{a_{\rho}}\sum_{w\in
W}\varepsilon(w)\mathrm{P}_{p,q}(\beta)e^{w(\mu+\rho+\beta)}=\sum_{\beta\in
Q_{+}}\mathrm{P}_{p,q}(\beta)\frac{a_{\mu+\beta+\rho}}{a_{\rho}}=\sum
_{\beta\in Q_{+}}\mathrm{P}_{p,q}(\beta)s_{\mu+\beta}.
\]
Now, for any $\beta\in Q_{+}$ we have $s_{\mu+\beta}=0$ or their exists an
element $w$ in $W$ and a dominant weight $\nu$ such that $\nu+\rho=w^{-1}%
(\mu+\beta+\rho)$ and $s_{\mu+\beta}=\varepsilon(w)s_{\nu}$. Then, we have
$\beta=w(\nu+\rho)-\rho-\mu$. By gathering the contributions of the same
character $s_{\nu}$, one obtains
\[
Q_{\mu}^{\prime}(p,q)=\sum_{w\in W}\varepsilon(w)\mathrm{P}_{p,q}(w(\nu
+\rho)-\rho-\mu)s_{\nu}=\sum_{\nu\in P_{+},\mu\leq\nu}K_{\nu,\mu}(p,q)s_{\nu}%
\]
as expected because $K_{\nu,\mu}(p,q)$ is nonzero only when $\mu\leq\nu$.
\end{proof}

\bigskip

We can then define the basis $\{P_{\nu}(p,q)\mid\nu\in P_{+}\}$ as the dual
basis of $\{Q_{\mu}^{\prime}(p,q)\mid\mu\in P_{+}\}$ for the scalar product
$\langle\cdot,\cdot\rangle_{(p,q)}$. We then have%
\[
s_{\nu}=\sum_{\mu\in P_{+}}K_{\nu,\mu}(p,q)P_{\mu}(p,q).
\]
When $p=q$, the polynomials $P_{\mu}(q)$ are the Hall-Littlewood polynomials.
They are deeply connected with the theory of affine Hecke algebras (see
\cite{NR} for a detailed review). In particular, by using the Satake
isomorphism, one can prove that the polynomials $K_{\nu,\mu}(q,q)=K_{\nu,\mu
}(q)$ are affine Kazhdan-Lusztig polynomials. For the double deformations
introduced here, the connection with the theory of affine Hecke algebras
cannot be obtained similarly.\ There indeed exist unequal parameters affine
Hecke algebras constructed from weight functions on the set of positive roots.
Nevertheless, these weight functions should be constant on the roots which
belong to the same orbit under the action of $W.$ This is not always the case
in our setting since $\overline{R}_{+}$ is not stable under the action of $W$.

\bigskip

\noindent\textbf{Acknowledgement}: The author thanks the Vietnam Institute for
Advanced Study in Mathematics (VIASM) and the ICERM in Providence for their
warm hospitality during the writing of this note. The author is also partially
supported by the Agence Nationale de la Recherche funding ANR CORTIPOM
21-CE40-001. He also Thanks B. Ion, T. Gerber and C. Lenart for fruitful discussions.

\bigskip\noindent C\'{e}dric Lecouvey: Institut Denis Poisson Tours.

\noindent Universit\'{e} de Tours Parc de Grandmont, 37200 Tours, France.

\noindent{cedric.lecouvey@univ-tours.fr}

\bigskip\noindent2010 Mathematics Subject Classification. 05E10, 17B10

\end{document}